\documentclass[11pt,letterpaper,reqno]{amsart}
\usepackage{graphicx}
\usepackage{amsmath}
\usepackage{amssymb}
\usepackage{amsfonts}
\usepackage{amsrefs}
\usepackage{amsthm}
\usepackage{cancel}
\usepackage{enumerate}
\usepackage[mathscr]{eucal}

\newcommand{\R}{\mathbb{R}}

\newcommand{\Hscr}{{\mathscr{H}}}

\newcommand{\inv}{^{-1}}
\newcommand{\ol}{\overline}

\newcommand{\sm}{\setminus}

\renewcommand{\div}{\text{div}\,}

\newtheorem{thm}{Theorem}
\newtheorem{lemma}[thm]{Lemma}
\newtheorem{prop}[thm]{Proposition}
\newtheorem{cor}[thm]{Corollary}
\newtheorem{conj}[thm]{Conjecture}

\theoremstyle{definition}
\newtheorem{definition}[thm]{Definition}

\theoremstyle{remark}
\newtheorem*{remark}{Remarks}
\newtheorem*{ack}{Acknowledgements}
\newtheorem*{aside}{Aside}

\DeclareMathOperator{\grad}{grad}
\DeclareMathOperator{\Div}{div}

\begin{document}
\title[Invariants of the harmonic conformal class]{Invariants of the harmonic conformal class of an asymptotically flat manifold}
\date{October 20, 2010}
\author{Jeffrey Jauregui}
\address{Department of Mathematics\\
University of Pennsylvania\\
David Rittenhouse Lab.\\
209 South 33rd Street\\
Philadelphia, PA 19104-6395}
\email{jauregui@math.upenn.edu}
\begin{abstract}
Consider an asymptotically flat Riemannian manifold $(M,g)$ of dimension $n \geq 3$ with nonempty compact boundary.  We recall the harmonic conformal class $[g]_h$ of the metric, which consists of all conformal rescalings given by a harmonic function raised to an appropriate power.  The geometric significance is that every metric in $[g]_h$ has the same pointwise sign of scalar curvature.  For this reason, the harmonic conformal class appears in the study
of general relativity, where scalar curvature is related to energy density (c.f. \cite{bray_RPI}).

Our purpose is to introduce and study invariants of the harmonic conformal class.  These invariants are closely related
to constrained geometric optimization problems involving hypersurface area-minimizers and the ADM mass.  In the final section, we discuss
possible applications of the invariants and their relationship with zero area singularities and the positive mass theorem.
\end{abstract}

\maketitle

\section{Introduction}
Let $M$ be a smooth manifold of dimension $n \geq 3$, possibly with a smooth boundary $\partial M$.  
Recall that two Riemannian metrics $g$ and $g'$ on $M$ are \emph{conformal}
if there exists a smooth function $u>0$ on $M$ such that $\ol g = u^{\frac{4}{n-2}} g$
pointwise as quadratic functions on the fibers of $TM$.  Conformality is obviously an equivalence relation, with the equivalence class
of a metric $g$ called the \emph{conformal class} of $g$:
$$[g] = \left\{u^{\frac{4}{n-2}} g : u > 0 \text { is a smooth function on } M\right\}.$$
The conformal class is an indispensable object in geometric analysis.

In the proof of the Riemannian Penrose inequality \cite{bray_RPI}, Bray observed the following fact: the relation on 
Riemannian metrics $\sim$ defined by
$$ \ol g \sim g \quad \text{ if and only if } \quad \ol g = u^{\frac{4}{n-2}} g,\; u > 0 \text{ is smooth, and } \Delta u =0$$
is an equivalence relation.  Here, $\Delta = \Div \grad$ is the Laplacian operator on functions with respect to $g$.  Reflexivity
of $\sim$ is clear; symmetry and transitivity follow from the formula for conformal metrics $\ol g = u^{\frac{4}{n-2}} g$:
\begin{equation}
\label{eqn_conf_laplacian}
\Delta(u\phi) = u^{\frac{n+2}{n-2}} \ol \Delta (\phi) + \phi \Delta u ,
\end{equation}
for any smooth function $\phi$ on $M$, where $\ol \Delta$ is the Laplacian with respect to $\ol g$ (c.f. Lemma 2.1 of \cite{bray_lee}).  Formula (\ref{eqn_conf_laplacian}) also explains the exponent $\frac{4}{n-2}$: no such equivalence relation exists for other values 
of this exponent.

The $\sim$ equivalence class of a metric $g$ is called the \emph{harmonic conformal class} of $g$ (although later we will slightly refine
this definition).  The harmonic conformal class is intimately connected with scalar curvature.  For, if $\ol g = u^{\frac{4}{n-2}} g$ are conformal metrics,
the scalar curvatures $\ol R$ and $R$ are related by
$$\ol R = u^{-\frac{n+2}{n-2}}\left( - \frac{4(n-1)}{n-2} \Delta u + R u\right),$$
so we see that $\ol g \sim g$ if and only if  $\ol R = u^{-\frac{4}{n-2}} R$.  In particular, metrics $\ol g \sim g$ have the same pointwise
sign of scalar curvature.

If $(M,g)$ is compact and without boundary, the harmonic conformal class consists only
of the constant rescalings of $g$, as follows from the maximum principle.  

In the context of general relativity, a natural class of manifolds to study are asymptotically flat manifolds of nonnegative scalar curvature
that possess a compact boundary.  The harmonic conformal class is well-adapted to studying such spaces, since it preserves both asymptotic
flatness (under suitable restrictions) and the nonnegativity of scalar curvature.  In the literature, much emphasis is placed on manifolds whose
boundary consists of minimal surfaces (c.f. \cites{imcf, bray_RPI}, for instance), but we make no such restriction here.

Our purpose is to define and study objects canonically associated to the harmonic conformal class of
an asymptotically flat manifold with boundary.  The outline is as follows: in section \ref{sec_definitions} we standardize
our definitions of asymptotic flatness, ADM mass, and the harmonic conformal class.  Section \ref{sec_numerical} defines two real number invariants of the 
harmonic conformal class, called $I_1$ and $I_2$.  In section \ref{sec_mass_profile} we motivate a constrained optimization problem for the ADM mass 
and apply it to define a function $\mu: \R^+ \to \R$ that depends only on the harmonic conformal class.  The first main result, 
Theorem \ref{thm_mu}, is that $\mu$ is given by
an explicit formula involving $I_1$ and $I_2$.  

We introduce a function $\alpha: \R^+ \to \R^+$ in section \ref{sec_area_profile}, also
a harmonic conformal invariant, that is substantially more subtle than its counterpart, $\mu$.  Roughly, the value of $\alpha(A)$ is determined
by maximizing the least area needed to enclose the boundary among metrics in the harmonic conformal class that measure the boundary
area to be at most $A$.  This minimax-type definition is somewhat delicate -- no simple formula for $\alpha$ is expected, and even showing 
that the maximum is attained seems to require enlarging the harmonic conformal class to allow for weak boundary regularity of the 
conformal factors. With this additional flexibility, we study $\alpha$ and the properties of the metrics which optimize it in section \ref{sec_properties}.  The main result here is that assuming good regularity for these optimal metrics, the resulting manifolds are such that the boundaries are enclosed by an area-minimizing surface that ``almost'' has zero mean curvature.  
The final two sections consist of examples and some conjectured applications of the techniques developed herein.  Assuming a certain
extension of the Riemannian Penrose inequality, we establish inequalities between the numerical invariants $I_1$ and $I_2$ and the functions $\mu$
and $\alpha$.  These estimates are particularly relevant for the study of zero area singularities, which we recall.  In closing, we discuss a possible
generalization of the positive mass theorem that allows for metrics with certain types of singularities.


\begin{ack}
Most of the content of this paper was part of my thesis work, and I am very indebted to my advisor Hugh Bray for countless discussions 
and suggestions.  I would also like to thank Bill Allard, Graham Cox, Michael Eichmair, George Lam, and Mark Stern for helpful discussions.
\end{ack}

\section{Definitions}
\label{sec_definitions}
We will consider asymptotically flat manifolds, which are spaces that geometrically approach Euclidean space in a precise sense.
\begin{definition}
\label{def_AF}
A smooth, connected, Riemannian manifold $(M,g)$ (possibly with compact boundary) of dimension $n\geq 3$ is \textbf{asymptotically flat} (with one end) if 
\begin{enumerate}[(i)]
\item there exists a compact subset $K \subset M$ and a diffeomorphism $\Phi: M \sm K \to \R^n \sm \ol{B}$ (where $\ol B$ is a closed ball), and
\item in the coordinates $(x^1, \ldots, x^n)$ on $M \sm K$ induced by $\Phi$, the metric obeys the decay conditions:
\begin{align*}
 |g_{ij} - \delta_{ij}| &\leq \frac{c}{|x|^p},&
 |\partial_k g_{ij}| &\leq \frac{c}{|x|^{p+1}},\\
 |\partial_k \partial_l g_{ij}| &\leq \frac{c}{|x|^{p+2}},&
 |R| &\leq \frac{c}{|x|^q},
\end{align*}
for $|x|=\sqrt{(x^1)^2 + \ldots + (x^n)^2}$ sufficiently large and all $i,j,k,l=1,\ldots,n$, where $c>0$, $p > \frac{n-2}{2}$, and $q > n$ are constants,
$\delta_{ij}$ is the Kronecker delta, $\partial_k = \frac{\partial}{\partial x^k}$, and $R$ is the scalar curvature of $g$.
\end{enumerate}
Such $(x^i)$ are called \textbf{asymptotically flat coordinates}.
\end{definition}
Several other inequivalent definitions of asymptotic flatness appear in the literature.  Next, we recall the definition of the ADM mass \cite{adm}, a number associated
to any asymptotically flat manifold, which provides some measure of the rate at which the metric becomes flat near infinity.
\begin{definition}
The \textbf{ADM mass} of an asymptotically flat manifold $(M,g)$ is the number 
$$m_{ADM}(g) = \frac{1}{2(n-1)\omega_{n-1}} \lim_{r\to \infty} \sum_{i,j=1}^n\int_{S_r} \left(\partial_i g_{ij} - \partial _j g_{ii}\right)\frac{x^j}{r} dA$$
where $(x^i)$ are asymptotically flat coordinates, $S_r$ is the coordinate sphere $\{|x|=r\}$, and $\omega_{n-1}$ is the
area of the unit sphere in $\R^n$.
\end{definition}
Other conventions for the normalizing constant appear in the literature.  The fundamental work of Bartnik establishes that the limit
exists, is independent of the choice of asymptotically flat coordinates, and is therefore a geometric invariant of $(M,g)$ \cite{bartnik}.

For instance, let $m>0$, and consider the following metric on $\R^n$ minus the open ball of radius $r_m:=\left(\frac{m}{2}\right)^{\frac{1}{n-2}}$ about the origin,
equipped with the metric
\begin{equation}
\label{eqn_schwarz_metric}
g_m = \left(1+\frac{m}{2|x|^{n-2}}\right)^{\frac{4}{n-2}} \delta,
\end{equation}
where $\delta$ is the standard flat metric. $g_m$ is called the \emph{Schwarzschild metric of mass $m$} and is 
asymptotically flat with ADM mass equal to $m$.

To define the harmonic conformal class, we slightly abuse our previous terminology by requiring the harmonic functions $u$ to approach 
one at infinity.  (Allowing $u$ to approach any positive constant at infinity only introduces constant rescalings of the metric.)
\begin{definition}
The \textbf{harmonic conformal class} of an asymptotically flat metric $g$ on a manifold $M$ of 
dimension $n\geq 3$ is the set of Riemannian metrics
$$[g]_h = \left\{ u^{\frac{4}{n-2}} g : u > 0 \text{ is smooth, } \Delta u = 0, \text{ and } u \to 1 \text{ at infinity}\right\}.$$
\end{definition}
For example, the Schwarzschild metric of mass $m$ and the flat metric on $\R^n$ minus the ball $B(0,r_m)$ belong to
the same harmonic conformal class.  In the case that $(M,g)$ is a complete asymptotically flat manifold
\emph{without} boundary, $[g]_h$ consists of the single element $g$.  (The maximum principle and the boundary
condition $u \to 1$ forces $u \equiv 1$.)

Throughout the rest of this paper, $(M,g)$ is an asymptotically flat manifold of dimension $n \geq 3$ with compact, smooth, 
nonempty boundary $\Sigma = \partial M$.  To be explicit, in the above definition we require that $u$ extends smoothly to the boundary.  It
is straightforward to check that every metric in $[g]_h$ is asymptotically flat, using the existence of an expansion of $u$ into
spherical harmonics near infinity \cite{bartnik}.

To conclude this section, we remark that $(M,g)$ has a unique Poisson kernel, namely a smooth, positive function $K(x,y)$, where 
$x \in M, y \in \Sigma$, and $x \neq y$, that is harmonic with respect to $g$ in the $x$-variable and satisfies the following property: if $w$ is a 
smooth, harmonic function on $M$ that approaches zero at infinity, then
$$w(x) = \int_\Sigma K(x,y) w(y) dA(y),$$
where $dA$ is the hypersurface measure on $\Sigma$ induced by $g$.  The Dirichlet problem for Laplace's equation can
be uniquely solved by prescribing continuous boundary data on $\Sigma$ and a constant at infinity.  The existence of the Poisson kernel follows from the existence of a Green's function, which in turn follows from asymptotic flatness.  

To simplify notation later, we define the following constants:
$$p = \frac{2(n-1)}{n-2}, \qquad k = \frac{4}{n-2}.$$

\section{Numerical invariants of the harmonic conformal class}
\label{sec_numerical}
In this section, we construct two \emph{numerical} invariants of the harmonic conformal class of a fixed asymptotically flat manifold
$(M,g)$ with compact boundary $\Sigma$.  The motivation for these natural invariants will come later in the paper.

First, recall that the \emph{capacity} of
$\Sigma$ in $(M,g)$, denoted $C_g(\Sigma)$, is the coefficient $c>0$ in the spherical harmonic expansion 
$$\varphi(x) = 1 - \frac{c}{|x|^{n-2}} + O(|x|^{1-n})$$
of the unique harmonic function $\varphi$ that vanishes on $\Sigma$ and approaches one at infinity.  An explicit formula for the capacity is
$$C_g(\Sigma) = \frac{1}{(n-2)\omega_{n-1}} \lim_{r \to \infty} \int_{S_r} \partial_\nu \varphi \, dA =: \frac{1}{(n-2)\omega_{n-1}} \int_{S_\infty} \partial_\nu \varphi \, dA,$$
where the last equality defines the notation $S_\infty$. Here, $\nu$ is the unit normal to the indicated surface pointing toward infinity,
$\partial_\nu$ is the directional derivative, and $dA$ is the area form on the indicated surface, all with respect to $g$. 
Since $\varphi$ is harmonic, the divergence theorem shows
$$\int_{S_\infty} \partial_\nu \varphi \, dA = \int_{S_r} \partial_\nu \varphi \, dA = \int_\Sigma \partial_\nu \varphi \, dA,$$
for all coordinate spheres $S_r$.  We follow the convention that the normal $\nu$ to $\Sigma$ also points toward infinity (into the manifold).  
By the maximum principle, $\partial_\nu \varphi > 0$ on $\Sigma$ or $S_r$.
\begin{lemma}
\label{lemma_I1}
If $(M,g)$ is asymptotically flat, then the number
$$I_1:=m_{ADM}(g) - 2C_g(\Sigma)$$
is an invariant of the harmonic conformal class $[g]_h$.  That is, if $\ol g \in [g]_h$, then
$$m_{ADM}(\ol g) - 2C_{\ol g}(\Sigma)=m_{ADM}(g) - 2C_g(\Sigma).$$
\end{lemma}
Note that $I_1$, the difference of ADM mass and twice the capacity, can be positive, negative, or zero.
\begin{proof}
From the definition of ADM mass and asymptotic flatness, it readily follows that the ADM masses of $\ol g=u^{k} g$ and $g$ (where $k=\frac{4}{n-2}$) are related by
\begin{equation}
\label{eqn_conf_mass_infinity}
m_{ADM}(\ol g) = m_{ADM}(g) - \frac{2}{(n-2)\omega_{n-1}} \int_{S_\infty} \partial_\nu u\, dA.
\end{equation}
The fact that the ADM mass minus the capacity is an invariant of $[g]_h$ then follows from the formula
\begin{equation}
\label{eqn_conf_capacity}
C_{\ol g}(\Sigma) = C_g(\Sigma) - \frac{1}{(n-2)\omega_{n-1}} \int_{S_\infty} \partial_{\nu} u dA,
\end{equation}
which we now prove.  Let $\varphi$ and $\ol \varphi$
be the harmonic functions (with respect to $g$ and $\ol g$, respectively) that vanish on $\Sigma$ and approach one at infinity.  Using formula (\ref{eqn_conf_laplacian}), one can check that $\varphi/u$ is harmonic with respect to $\ol g$, is zero on $\Sigma$, and approaches one at infinity.
Therefore by uniqueness, $\ol \varphi = \varphi/u$, so
\begin{align*}
 (n-2) \omega_{n-1} C_{\ol g}(\Sigma) &= \int_{S_\infty} \partial_{\ol \nu} (\varphi/u) \, \ol {dA},
\end{align*}
where $\ol{dA}$ and $\ol \nu$ are hypersurface measure and the unit normal with respect to $\ol g$.  Since $u \to 1$ and $\varphi \to 1$ at infinity, we have
\begin{align*}
 (n-2) \omega_{n-1} C_{\ol g}(\Sigma) &= \int_{S_\infty} \partial_{\nu} \left(\varphi/u\right) dA\\
	&= \int_{S_\infty} \partial_{\nu} \varphi dA - \int_{S_\infty} \partial_{\nu} u dA\\
	&=  (n-2) \omega_{n-1} C_{g}(\Sigma) - \int_{S_\infty} \partial_{\nu} u dA,
\end{align*}
proving (\ref{eqn_conf_capacity}).
\end{proof}

\begin{lemma}
\label{lemma_I2}
Suppose $(M,g)$ is asymptotically flat, of dimension $n \geq 3$.  Let $\varphi$ 
be the harmonic function (with respect to $g$) that vanishes on $\Sigma$ and approaches one at infinity.  Then the number
$$I_2:= \frac{2}{(n-2)^2} \left(\frac{1}{\omega_{n-1}} \int_\Sigma (\partial_\nu \varphi)^{\frac{2(n-1)}{n}} dA\right)^\frac{n}{n-1}$$
is an invariant of the harmonic conformal class of $g$.  That is, if $\ol g \in [g]_h$ with hypersurface measure $\ol{dA}$, unit normal
$\ol \nu$ to the boundary, and harmonic function $\ol \varphi$ that vanishes on $\Sigma$ and approaches one at infinity, then
\begin{equation}
\int_\Sigma (\partial_{\ol \nu} \ol \varphi)^{\frac{2(n-1)}{n}} \ol{dA} = \int_\Sigma (\partial_\nu \varphi)^{\frac{2(n-1)}{n}} dA.
\end{equation}
\end{lemma}
\begin{proof}
Suppose $\ol g = u^{k} g$ belongs to $[g]_h$.  As explained in the proof of Lemma \ref{lemma_I1}, we have the equality
$\ol \varphi = \varphi/u$.  Since lengths with respect to $\ol g$ and $g$ are related by a factor of $u^{\frac{2}{n-2}}$ pointwise, we see
$$\ol \nu = u^{-\frac{2}{n-2}} \nu,$$
and using the fact that $\varphi$ vanishes on $\Sigma$,
\begin{align*}
\partial_{\ol \nu} \ol \varphi &= u^{-\frac{2}{n-2}} \partial_\nu (\varphi/u)\\
	&= u^{-\frac{2}{n-2}} \frac{\partial_\nu \varphi}{u}\\
	&= u^{-\frac{n}{n-2}} \partial_\nu \varphi.
\end{align*}
Next, the hypersurface measures are related by $\ol{dA} = u^{\frac{2(n-1)}{n-2}} dA$, and it readily follows that 
$$(\partial_{\ol \nu} \ol \varphi)^{\frac{2(n-1)}{n}} \ol{dA} = (\partial_\nu \varphi)^{\frac{2(n-1)}{n}} dA$$
as measures on $\Sigma$.  In particular, the integrals over $\Sigma$ of these measures agree.
\end{proof}

\section{The mass profile function}
\label{sec_mass_profile}
In this section and the next, we fix $(M,g)$ as above and consider some problems that involve minimizing or maximizing certain geometric 
quantities within the harmonic conformal class of $g$.  Working with asymptotically flat manifolds, it is natural to consider the ADM mass as a geometric quantity to be optimized.  Recall that if $\ol g=u^{k} g$ belongs to $[g]_h$, then formula (\ref{eqn_conf_mass_infinity})
relates the ADM masses of $\ol g$ and $g$.  Using the fact that $u$ is harmonic with respect to $g$, we also see
\begin{equation}
\label{eqn_conf_mass}
m_{ADM}(\ol g) = m_{ADM}(g) - \frac{2}{(n-2)\omega_{n-1}} \int_{S_r} \partial_\nu u\, dA,
\end{equation}
for any coordinate sphere $S_r$.  The last term, including the minus sign, can be interpreted as twice the coefficient $a$ in the expansion of $u$ into spherical harmonics
for $|x|$ large:
$$u(x) = 1 + \frac{a}{|x|^{n-2}} + O(|x|^{1-n}).$$
 
From (\ref{eqn_conf_mass}), it is not difficult to see that the ADM mass can be made arbitrarily large for metrics $\ol g = u^{k}g$ in $[g]_h$ by choosing a harmonic conformal
factor $u$ that is large on $\Sigma$.  For such $u$, $\Sigma$ clearly has large $\ol g$-area.  
This motivates the question: how large can the ADM mass be made among metrics in $[g]_h$ that have a fixed
upper bound $A$ on the area of the boundary $\Sigma$?  Let $|\Sigma|_g$ denote the area (hypersurface measure)
of $\Sigma$ with respect to a metric $g$. We have the following definition.
\begin{definition}
\label{def_mu}
Given a number $A>0$, define
$$\mu(A) = \sup_{\ol g \in [g]_h} \{m_{ADM}(\ol g) : \; |\Sigma|_{\ol g} \leq A\}.$$
\end{definition}
The number $\mu(A)$ is the largest possible value of the ADM mass among metrics in $[g]_h$ that measure the boundary area to be at most $A$.
Below we will see that $\mu(A)$ is finite for each value of $A$, so in particular $\mu$ is
well-defined as a function $\R^+ \to \R$.  But first we make the observation:
\begin{lemma}
The function $\mu: \R^+ \to \R$ is independent of the choice of metric in $[g]_h$.
\end{lemma}
So we say that $\mu$ is an invariant of the harmonic conformal class of $g$ and call $\mu$ the \emph{mass profile function} of $[g]_h$.
The proof of the lemma is trivial: $\mu$ is formed by maximizing a geometric quantity
(ADM mass) subject to a geometric constraint (upper bound for area) over the whole harmonic conformal class.

Before moving on, we remark that $\mu$ may be interpreted purely in terms of the behavior of harmonic functions on $M$:
\begin{align*}
 \mu(A) &= \sup_{u} \left\{m_{ADM}(g) - \frac{2}{(n-2)\omega_{n-1}} \int_{S_\infty} \partial_\nu u\, dA :\right.\\
	&\qquad \qquad \left. u>0,\Delta u = 0, u \to 1 \text{ at infinity, and } \int_\Sigma u^{p} dA \leq A\right\},
\end{align*}
where $p=\frac{2(n-1)}{n-2}$.  That is, $\mu(A)$ is essentially found by maximizing the coefficient of the $\frac{1}{|x|^{n-2}}$ term in the expansion at infinity of positive harmonic functions
$u$ with an $L^p$ upper bound for $u|_\Sigma$.  However, from this viewpoint it is not transparent that $\mu$ is an invariant of $[g]_h$.

In the following theorem we give a complete understanding of $\mu(A)$ by proving an explicit formula in terms of the numerical 
invariants $I_1$ and $I_2$ introduced above.  In the course of the proof, we show that given $A$, there exists a unique metric
in $[g]_h$ attaining the supremum for $\mu(A)$.
\begin{thm}
\label{thm_mu}
For $A>0$, we have the formula
$$\mu(A) = I_1 + \left(2I_2\right)^{1/2} \left(\frac{A}{\omega_{n-1}}\right)^{1/p}.$$
In particular, $\mu:\R^+ \to \R$ is a smooth, increasing function that is bounded below and unbounded above.  Moreover, the function
$\mu$ is completely determined by the numerical invariants $I_1$ and $I_2$ defined in Lemmas \ref{lemma_I1} and \ref{lemma_I2}.
\end{thm}
\begin{proof}
Fix $A>0$, and suppose initially that $\ol g=u_0^{k} g$ attains the supremum in the definition of $\mu(A)$ and satisfies
the area bound $|\Sigma|_{\ol g} \leq A.$  We claim that $|\Sigma|_{\ol g} = A$. For, if $|\Sigma|_{\ol g} < A$, then by adding
a small constant to boundary data for $u_0$, we could construct a metric in $[g]_h$ with boundary area equal to $A$.  By the maximum
principle and formula (\ref{eqn_conf_mass_infinity}), the ADM mass of this new metric would exceed that of $\ol g$, so that the latter could
not attain the supremum for $\mu(A)$.  So $|\Sigma|_{\ol g} = A$.

We show that $u_0$ satisfies a variational principle.  For $t \in (-\epsilon,\epsilon)$,
let $u_t$ be a smoothly-varying path in the space of positive harmonic functions on $M$, passing through $u_0$ at $t=0$, such that
\begin{enumerate}[(i)]
\item for each $t$, $u_t$ approaches one at infinity, and
\item $\int_{\Sigma} u_t^{p} dA \equiv A$ for all $t$.
\end{enumerate}
This is equivalent to stating that $u_t^{k} g$ is a smooth path of metrics in $[g]_h$ that fixes the boundary area at 
the value $A$.  Since $u_0^{k} g$ maximizes $\mu(A)$ among metrics with
boundary area $A$, the smooth function
$$t \mapsto m_{ADM}\left(u_t^{k} g\right)$$
has a local maximum at $t=0$.  Using formula (\ref{eqn_conf_mass}), we have 
\begin{align*}
0 &=\frac{d}{dt} m_{ADM}\left(u_t^{k} g\right)\Big|_{t=0}\\
&=\frac{d}{dt} \left(m_{ADM}(g) - \frac{2}{(n-2)\omega_{n-1}} \int_{S_r} \partial_\nu u_t dA\right)\Big|_{t=0}\\
&= -\frac{2}{(n-2)\omega_{n-1}} \int_{S_r} \partial_\nu w_0  dA,
\end{align*}
where $w_0(x) = \frac{d}{dt}u_t(x)\big|_{t=0}.$  Also, since the boundary area is constant in $t$,
$$0 = \frac{d}{dt}|\Sigma|_{u_t^{k} g}\Big|_{t=0} = \frac{d}{dt}\int_\Sigma u_t^{p} dA\Big|_{t=0} = p \int_\Sigma u_0^{\frac{n}{n-2}} w_0 dA.$$
Observe that $w_0$ is a smooth harmonic function on $M$ that approaches zero at infinity.  Given any such $w_0$
satisfying $\int_\Sigma u_0^{\frac{n}{n-2}} w_0 dA=0$, it is possible to construct a smooth family of harmonic functions $u_t$ with $\frac{d}{dt} u_t\big|_{t=0} = w_0$ satisfying properties (i) and (ii) above.  We now see that the unknown harmonic function $u_0$ (if it exists) satisfies the statement:
\begin{quotation}
If $w_0$ is any harmonic function on $M$ that approaches zero at infinity, and 
if $\int_\Sigma u_0^{\frac{n}{n-2}} w_0 dA = 0$, then $\int_{S_r} \partial_\nu w_0 dA=0$ as well.
\end{quotation}
To make this more concrete, Lemma \ref{lemma_compute_term} below shows how to compute $\int_{S_r} \partial_\nu w_0 dA$ solely from the boundary data for $w_0$: there exists a smooth,
positive function $V$ on $\Sigma$ such that for all harmonic functions $w_0$ on $M$ that approach zero at infinity, we have
$$- \frac{1}{(n-2)\omega_{n-1}} \int_{S_r} \partial_\nu w_0 dA = \int_\Sigma V w_0 dA.$$
Writing $\psi = {w_0}|_\Sigma$ we see that $u_0$ satisfies the property: 
\begin{equation}
\text{If $\psi$ is a smooth function on $\Sigma$ with $\int_\Sigma u_0^{\frac{n}{n-2}} \psi dA = 0$,
then $\int_\Sigma V\psi dA =0$.}
\label{eqn_critical_property}
\end{equation}
The trick is to utilize this observation to determine what $u_0$ should be: we will \emph{define} $u_0$ to be harmonic, one at infinity with boundary data given by a constant times $V^{\frac{n-2}{n}}$, so that $u_0$ satisfies (\ref{eqn_critical_property}) automatically.  Specifically, let $f_0$ be the function on $\Sigma$ given by:
$$f_0(x)=\frac{A^{\frac{n-2}{2(n-1)}}}{\left(\int_\Sigma V^{\frac{2(n-1)}{n}} dA\right)^{\frac{n-2}{2(n-1)}}}V(x)^{\frac{n-2}{n}},$$
and let $u_0$ be harmonic, one at infinity, with boundary data $f_0$.  Then the metric $u_0^{k} g \in [g]_h$ has boundary area equal to $A$ and is, by the above computations, a critical point for the ADM mass among metrics in $[g]_h$ with boundary area equal to $A$.

The next step is to show $u_0^{k} g$ indeed attains the supremum for $\mu(A)$.  We do so by showing the ADM mass satisfies a concavity property on paths of metrics in $[g]_h$ that fix the boundary area.  Let $f_1$ be any smooth, positive function on $\Sigma$ distinct from $f_0$ that serves as boundary data for a harmonic function $u_1$ that approaches one at infinity.  Assume $\int_\Sigma f_1^{p} dA = A$. 
To consider a path between $f_0$ and $f_1$ that fixes the $L^{p}$ norm, we define for $t \in [0,1]$:
$$f_t(x) = \frac{A^{\frac{n-2}{2(n-1)}}}{\|(1-t)f_0 + t f_1\|_{L^{p}}} \big[(1-t)f_0(x) + tf_1(x)\big].$$
Let $u_t$ be the harmonic function, one at infinity, with boundary data $f_t$, so $u_t^{k} g$ is a path in $[g]_h$ for $t \in[0,1]$
that has boundary area $A$ for all $t$.  By convexity of the $L^{p}$ norm, we have
\begin{equation}
\label{eqn_convexity}
f_t(x)  \geq (1-t)f_0(x) + tf_1(x).
\end{equation}
Then certainly
$$\int_\Sigma Vf_t dA  \geq (1-t) \int_\Sigma V f_0 dA + t\int_\Sigma V f_1 dA.$$
Applying Lemma \ref{lemma_compute_term} to each of these three terms, adding $\frac{1}{2}m_{ADM}(g)$ to both sides, 
and using formula (\ref{eqn_conf_mass_infinity}), we have
\begin{equation}
\label{eqn_concavity}
m_{ADM}\left(u_t^{k} g\right) \geq (1-t)\,m_{ADM}\left(u_0^{k} g\right) + t\,m_{ADM}\left(u_1^{k} g\right).
\end{equation}
Note that the inequality is strict for $t \in (0,1)$ (because (\ref{eqn_convexity}) is) and equality holds at the endpoints.    

If $u_0^{k} g$ were not a maximizer for $\mu(A)$, then there exists $u_1^{k} g \in [g]_h$ with boundary area $A$ and
$$m_{ADM}(u_1^{k} g) > m_{ADM}(u_0^{k} g).$$
Differentiating (\ref{eqn_concavity}) at $t=0$ we see that
$$\frac{d}{dt} m_{ADM}(u_t^{k} g)\big|_{t=0} > 0,$$
contradicting the fact that $u_0^{k} g$ is a critical point for the ADM mass among metrics in $[g]_h$ with the same boundary area.  This shows that $u_0^{k} g$ is indeed a maximizer for $\mu(A)$, and (\ref{eqn_concavity}) also shows that the maximizer is unique.

Finally, we compute $\mu(A)$ as the ADM mass of $u_0^{k}g$ and simplify.  Using formula (\ref{eqn_conf_mass_infinity}) again, as well
as Lemma \ref{lemma_compute_term}, 
\begin{align*}
\mu(A) &= m_{ADM}(u_0^{k} g)\\
	&= m_{ADM}(g) - \frac{2}{(n-2)\omega_{n-1}} \int_{S_\infty} \partial_\nu u_0 dA\\
	&= m_{ADM}(g) + 2\int_\Sigma V (u_0 - 1) dA,
\end{align*}
since $u_0-1$ is harmonic, approaching zero at infinity.  
The integral $\int_\Sigma V u_0 dA$ can be simplified based on our definition of $f_0 = u_0|_\Sigma$.  It is straightforward to check
that
$$\int_\Sigma V u_0 dA = A^{\frac{n-2}{2(n-1)}} \left(\int_\Sigma V^{\frac{2(n-1)}{n}} dA\right)^{\frac{n}{2(n-1)}}.$$
Also, applying Lemma \ref{lemma_compute_term} to the harmonic function $1-\varphi$ that is one on $\Sigma$ and zero at infinity,
we see
$$\int_\Sigma V dA = C_g(\Sigma).$$
Putting it all together, we have
$$\mu(A) = m_{ADM}(g) - 2C_g(\Sigma) + 2A^{\frac{n-2}{2(n-1)}} \left(\int_\Sigma V^{\frac{2(n-1)}{n}} dA\right)^{\frac{n}{2(n-1)}}.$$
Using the definition of $I_1$ and $I_2$, as well as the fact that $V = \frac{1}{(n-2)\omega_{n-1}} \partial_\nu \varphi$ (see
the proof of Lemma \ref{lemma_compute_term}), we obtain
$$\mu(A) = I_1 + \left(2I_2\right)^{1/2} \left(\frac{A}{\omega_{n-1}}\right)^{\frac{n-2}{2(n-1)}}.$$
\end{proof}

\begin{lemma}
\label{lemma_compute_term}
Suppose $(M,g)$ is asymptotically flat with boundary $\Sigma$.  Then there exists a unique smooth, positive function $V$ on $\Sigma$ such that
for all harmonic functions $w$ that approach zero at infinity,
$$- \frac{1}{(n-2)\omega_{n-1}} \int_{S_\infty} \partial_\nu w dA = \int_\Sigma V w dA.$$
\end{lemma}
In other words, $V$ represents the linear functional on $C^\infty(\Sigma)$ that maps the boundary values of any harmonic function $w$ vanishing at infinity
to the coefficient of the $\frac{1}{|x|^{n-2}}$ term in the expansion of $w$ into spherical harmonics.  Note that
since $w$ is harmonic, the integral over $S_\infty$ may be replaced with the integral over any coordinate sphere
$S_r$.
\begin{proof}
Uniqueness is readily verified; we proceed to derive a formula for $V$.  Let $w$ be a harmonic function with smooth boundary data that approaches zero at infinity.  Let $\varphi$ be 
harmonic, zero on $\Sigma$ and
one at infinity.  Then by harmonicity we have
\begin{align*}
 0 &= \int_M \big[\div(w \nabla \varphi) - \div(\varphi \nabla w)\big] dV,
\end{align*}
where $\div, \nabla$ and $dV$ are the divergence, gradient, and volume measure with respect to $g$.  By the divergence theorem, this becomes
\begin{align*}
 0 &= \int_{S_\infty} w \partial_\nu \varphi \, dA - \int_{\Sigma} w \partial_\nu \varphi \, dA 
	- \int_{S_\infty} \varphi \partial_\nu w dA + \int_{\Sigma} \varphi \partial_\nu w  dA.
\end{align*}
Using the fact that $w$ approaches zero at infinity and $\varphi$ vanishes on $\Sigma$ and approaches one at infinity, we conclude
$$-\int_{S_\infty} \partial_\nu w dA = \int_{\Sigma} w \partial_\nu \varphi \, dA.$$
Choosing $V = \frac{1}{(n-2)\omega_{n-1}}\partial_\nu \varphi$, the proof is complete.
\end{proof}

\begin{aside}
After having studied the function $\mu(A)$, it is natural to consider the problem of \emph{minimizing} the ADM mass within $[g]_h$
among metrics having a fixed \emph{lower bound} on the boundary area:
$$\underline{\mu}(A)=\inf_{\ol g \in [g]_h} \{m_{ADM}(\ol g) : \; |\Sigma|_{\ol g} \geq A\}.$$
This infimum defines a function of $A$ that is an invariant of $[g]_h$.  We claim that $\underline{\mu}(A)$ is the constant function equal to $I_1$.  First, suppose
$\ol g=u^{k} g \in [g]_h$.  From formula (\ref{eqn_conf_mass}) and the maximum principle, we see
\begin{align*}
m_{ADM}(\ol g) 
&\geq m_{ADM}(g) -  \frac{2}{(n-2) \omega_{n-1}} \int_{S_r} \partial_\nu (\varphi) dA,
\end{align*}
which equals $I_1$, so $\underline{\mu}(A) \geq I_1$.  To prove equality, it is enough to find a sequence $\{u_i^{k} g\}$ in $[g]_h$ such that each metric in the sequence has boundary area $A$ and such that $u_i$ converges pointwise to $\varphi$ on the interior of $M$.  Then by harmonicity, this would then show that
$\int_{S_r} \partial_\nu (u_i) dA \to \int_{S_r} \partial_\nu (\varphi) dA.$  We leave it to the reader to construct such a sequence by letting
$u_i$ have boundary data $f_i$, where $f_i^{p}$ approximates a Dirac delta of measure $A$ as $i \to \infty$.
\end{aside}

\section{The area profile function}
\label{sec_area_profile}
In this section we introduce a function $\alpha:\R^+ \to \R^+$ whose definition formally appears similar to that of the mass profile 
function $\mu$; the idea is still to maximize a geometric quantity over $[g]_h$ subject to a boundary area constraint.  We will see, however, that $\alpha$ is much more subtle than $\mu$ in a number of respects.

Given an asymptotically flat manifold $(M,g)$, a \emph{surface} $S$ \emph{enclosing} the boundary $\Sigma$ is a set that differs 
(in the sense of integral currents) from $\Sigma$ by the boundary of a bounded region $\Omega \subset M$ with $(n-1)$-rectifiable boundary of finite 
$\Hscr^{n-1}$ measure:
$$S = \Sigma + \partial \Omega.$$
Here, $\Hscr^{n-1}$ is Hausdorff $(n-1)$-measure with respect to $g$; we also use the notation
$|S|_g=\Hscr^{n-1}(S)$, which we call the \emph{area} of $S$.
We define the \emph{minimal enclosing area} of the boundary $\Sigma$ with respect to $g$ to be the number:
\begin{equation}
\label{eqn_min_enclosing_area}
\min(\Sigma,g) = \inf_S \left\{|S|_g \; : \; S \text{ is a surface enclosing } \Sigma\right\}.
\end{equation}
From standard results in geometric measure theory, this infimum is attained by at least one surface (this uses
asymptotic flatness and the Federer--Fleming compactness theorem for integral currents; see \cite{simon} for instance).
Such a surface $S$ is called a \emph{minimal area enclosure} of $\Sigma$
with respect to $g$.  In low dimensions $3 \leq n \leq 7$, $S$ has $C^{1,1}$ regularity, and $S \sm \Sigma$, if non-empty, is a $C^\infty$ minimal (zero mean curvature) surface.
We remark that there exists a unique \emph{outermost} minimal area enclosure that we denote by $\tilde \Sigma_g$.    For more details on existence, uniqueness, and regularity see section 1 of \cite{imcf},
which uses the terminology of minimizing hulls.

The minimal enclosing area of $\Sigma$ is a geometric quantity, and we consider the problem of optimizing it within the harmonic conformal class.
Like the ADM mass, the number $\min(\Sigma,g)$ takes on arbitrarily large values within $[g]_h$ -- but this is
not true if we restrict to metrics in $[g]_h$ with an upper bound for the boundary area. This motivates the preliminary definition: for $A>0$, let
\begin{equation}
\label{eqn_alpha_0}
\alpha_{\circ}(A) = \sup_{\ol g \in [g]_h} \{\min(\Sigma, \ol g) : \; |\Sigma|_{\ol g} \leq A\}.
\end{equation}
In words, $\alpha_{\circ}(A)$ is the maximum possible value of the minimal enclosing area for metrics in $[g]_h$ that have boundary area at most $A$.  In sections \ref{sec_properties} and \ref{sec_applications} we will give more geometric motivation for why optimizing the minimal area enclosure is of interest.

It is clear that $\alpha_{\circ}$ defines a function $\R^+ \to \R^+$ that is nondecreasing and satisfies $\alpha_{\circ}(A) \leq A$.  
(To see the last point, note that if the boundary area is at most $A$, then the least area needed to enclose the boundary is certainly no more than $A$.)
What is far from clear is that the supremum for $\alpha_{\circ}(A)$ is attained.  A few essential differences between the functions $\alpha_{\circ}$ and $\mu$
now come to the surface:
\begin{enumerate}[(i)]
 \item Suppose $\ol g=u_0^{k}g \in [g]_h$ attains the supremum for $\alpha_{\circ}(A)$.  Consider a smooth path $\ol g_t = u_t^{k}g$
in $[g]_h$ passing through $\ol g$ at $t=0$ and preserving the condition that the boundary area is at most $A$.  In general, particularly in the case
in which there exist multiple minimal area enclosures of $\Sigma$, it is entirely possible that the function
\begin{equation}
\label{eqn_min_t}
t \mapsto \min(\Sigma, \ol g_t)
\end{equation}
is not differentiable at $t=0$.  In other words, a maximizing metric $\ol g$ does not obviously satisfy a variational property.  
\item Even in the case in which $\frac{d}{dt}\big|_{t=0} \min(\Sigma, \ol g_t) = 0$, 
the resulting variational statement is not particularly useful: it leads to a statement on the behavior of harmonic functions restricted to the 
surface $\tilde \Sigma_{\ol g}$, rather than the boundary.  Moreover, it seems
highly unlikely that a formula for $\alpha_{\circ}(A)$ could be given in terms of the numerical invariants $I_1$ and $I_2$.  This all contrasts sharply with the 
case of $\mu(A)$ (c.f. Theorem \ref{thm_mu}).
 \item Concavity: in the proof of Theorem \ref{thm_mu}, we saw the ADM mass satisfies a concavity property that allowed us to show that a critical point was
	necessarily a global maximum.  The minimal enclosing area evidently satisfies no such property, essentially for the reason that the area of a surface
	$S$ with respect to $u^{k} g$ is a \emph{convex} function of $u$.
\end{enumerate}
Despite these difficulties, we would still like to maximize the minimal enclosing area within $[g]_h$ in the same spirit as in (\ref{eqn_alpha_0}).  To carry this out, we will 
enlarge the space $[g]_h$ so as to obtain a space with a useful compactness property, a step that was unnecessary for $\mu(A)$.

\subsection{The generalized harmonic conformal class}
For the purposes of maximizing $\min(\Sigma, \ol g)$ as in equation (\ref{eqn_alpha_0}), 
we enlarge the set $[g]_h$ as follows by allowing metrics with weaker boundary regularity. 
Let $f \geq 0$ belong to $L^{p}(\Sigma)$ (with respect to the hypersurface measure $dA$ induced by $g$).  
The reason for considering $L^{p}$ with $p=\frac{2(n-1)}{n-2}$
is that for smooth conformal metrics $\ol g=u^{k}g$, the area measures on hypersurfaces are related 
by $\ol{dA} = u^{p} dA.$
For $x$ in the interior of $M$, define
\begin{equation}
\label{eqn_assoc_harmonic}
u(x) = \varphi(x) + \int_\Sigma K(x,y) f(y) dA(y),
\end{equation}
where $K(x,y)$ is the Poisson kernel for $(M,g)$ (where $x \in M$, $y \in \Sigma$ and $x \neq y$), and $\varphi(x)$ is the unique $g$-harmonic
function that vanishes on $\Sigma$ and approaches one at infinity (c.f. section \ref{sec_definitions}).
In particular, $u$ is $g$-harmonic in the interior of $M$ and tends to one at infinity.  We call $u$ the \emph{harmonic function associated to} 
$f$.  Since $f$ is determined uniquely by $u$ (up to almost-everywhere equivalence), we also say $f$ is the function in 
$L^{p}(\Sigma)$ \emph{associated to} $u$.

\begin{remark}
\label{remarks_harmonic}
The function $u$ defined in (\ref{eqn_assoc_harmonic}) is smooth and positive in $M\sm \Sigma$, but need not extend continuously to $\Sigma$.
Also, while it is not clear that the trace of $u$ onto $\Sigma$ is defined (in the sense of Sobolev spaces), it is the case that for 
almost all $y \in \Sigma$, given a path $\gamma:[0,\epsilon) \to M$ with $\gamma(0)=y$ and $\gamma'(0)$ transverse to $\Sigma$, 
$u \circ \gamma(t)$ converges to $f(y)$ as $t\to 0^+$. This is true essentially because for $t>0$, the function $y \mapsto K(\gamma(t),y)$ 
forms an approximation to the identity on $\Sigma$, based at $y$, in the limit $t \to 0^+$.
\end{remark}
On $M \sm \Sigma$, $u^{k} g$ is a smooth Riemannian metric, and we make the following definition.
\begin{definition}
\label{def_ghcc}
The \textbf{generalized harmonic conformal class} of $g$ is the set $\ol{[g]_h}$ of all Riemannian metrics $u^{k} g$ on $M \sm \Sigma$, where
$u$ is the harmonic function associated to some nonnegative $f \in L^{p}(\Sigma)$ as in (\ref{eqn_assoc_harmonic}).
\end{definition}
We may think of $\ol{[g]_h}$ as the closure of $[g]_h$ with respect to the $L^{p}(\Sigma)$ norm.  
We remark that both of the sets $\ol {[g]_h}$ and $L^{p}(\Sigma)$ are 
unchanged if $g$ is replaced by some other metric in $[g]_h$, and that
$\ol{[g]_h}$ is (non-canonically) bijective to the set of nonnegative functions in $L^{p}(\Sigma)$ via
(\ref{eqn_assoc_harmonic}).
Sometimes for emphasis we will refer to $[g]_h$ as the \emph{smooth} harmonic conformal class of $g$.

We also point out that every metric in $\ol{[g]_h}$ is asymptotically flat as in Definition \ref{def_AF}, 
modulo the implicit assumption of smoothness up to the boundary.  To
define the area of any surface $S$ enclosing $\Sigma$ 
with respect to $\ol g=u^{k} g \in \ol{[g]_h}$, we decompose $S$ into the pieces $S \cap \Sigma$ and $S \sm \Sigma$:
\begin{equation}
|S|_{\ol g}= \int_{S \cap \Sigma} f^{p} d\Hscr^{n-1} + \int_{S \sm \Sigma} u^{p} d\Hscr^{n-1},
\label{eqn_def_area}
\end{equation}
where $f$ is the function in $L^{p}(\Sigma)$ associated to $u$, and $\Hscr^{n-1}$ is Hausdorff $(n-1)$-measure on $M$ with respect to $g$.  
Equation (\ref{eqn_def_area}) is well-defined (dependent only on $\ol g$ and not on $g$) and
reproduces the usual notion of area in the case that $\ol g$ is smooth.  


\subsection{The area profile function}
Using the generalized harmonic conformal class, we now define the area profile function.
\begin{definition}
\label{def_alpha}
Given numbers $A>0$ and $C\geq 1$, define
$$\alpha_C(A) = \sup_{\ol g=u^{k} g \in \ol{[g]_h}} \left\{\min(\Sigma, \ol g) : \; |\Sigma|_{\ol g} \leq A \text{ and } u \leq C\right\},$$
and let
$$\alpha(A) = \lim_{C \to \infty} \alpha_C(A).$$
We call $\alpha(A)$ the \textbf{area profile function}.
\end{definition}
In this definition, we use (\ref{eqn_def_area}) to define area with respect to metrics $\ol g$ and define
the minimal enclosing area $\min(\Sigma, \ol g)$ exactly as in (\ref{eqn_min_enclosing_area}).  The basic idea here is
to first maximize $\min(\Sigma,\ol g)$ with an area upper bound $A$ and requiring the conformal factors to be bounded by $C$,
then let $C$ go to infinity.  Observe that the limit
defining $\alpha(A)$ exists, since for fixed $A$, the quantity $\alpha_C(A)$ viewed as a function of $C$ 
is bounded above by $A$ (by the definition of minimal
enclosing area) and non-decreasing (by definition of supremum).  We require $C \geq 1$ since the function $u$ must approach
the value of 1 at infinity.  

We remark that it is possible to define an area profile function
by replacing $[g]_h$ with $\ol{[g]_h}$ in (\ref{eqn_alpha_0}), obtaining a harmonic conformal invariant, but we opt not to do so -- working
directly with conformal metrics with merely $L^{p}$ (rather than $L^\infty$) boundary data is difficult, and our definition of $\alpha(A)$ allows us to bypass this technical detail.

Evidently $\alpha_C(A)$ is not an invariant of $[g]_h$,
since the pointwise upper bound $u \leq C$ is not a geometric statement.  However, invariance is restored by taking the limit $C \to \infty$.
\begin{lemma}
The function $\alpha:\R^+ \to \R^+$ depends only on the harmonic conformal class $[g]_h$.
\end{lemma} 
\begin{proof}
Let $g_1, g_2$ be two metrics in the same harmonic conformal class, determining functions $\alpha^{(1)}_C(A)$ and $\alpha^{(2)}_C(A)$, 
respectively, as in Definition \ref{def_alpha} .  Say $g_2 = \psi^{k} g_1$, where $\psi$ is smooth, positive, and harmonic with respect to $g_1$, approaching one at infinity.  There exist 
positive constants $a,b$ such that
$a \leq \psi \leq b$ on $M$, since $\partial M$ is compact and $\psi \to 1$ at infinity.  Let $\ol g = u^{k} g_1$ be a valid test metric in $\ol{[g_1]_h}$ for $\alpha_C^{(1)}(A)$, meaning
$$|\Sigma|_{\ol g} \leq A \text{ and } u \leq C.$$
Then the same metric $\ol g$ can be written as $\left(\frac{u}{\psi}\right)^{k} g_2$, an element of $\ol{[g_2]_h}$ (by formula (\ref{eqn_conf_laplacian})) that satisfies
$$|\Sigma|_{\ol g} \leq A \text{ and } \frac{u}{\psi} \leq \frac{C}{a}.$$
Therefore $\ol g$ is a valid test metric for $\alpha^{(2)}_{C/a}(A)$.  
Since $g_1$ and $g_2$ determine the same generalized harmonic conformal class, we see that the set of test metrics for $\alpha^{(1)}_C(A)$ is a subset
of the set of test metrics for $\alpha^{(2)}_{C/a}(A)$.  In particular,
$$\alpha^{(1)}_C(A) \leq \alpha^{(2)}_{C/a}(A).$$
Taking the limit $C \to \infty$, then applying the same argument with the roles of $g_1$ and $g_2$ (and $a$ and $b$) swapped, we see
$$\lim_{C \to \infty} \alpha^{(1)}_C(A) = \lim_{C \to \infty} \alpha^{(2)}_C(A),$$
proving that $\alpha(A)$ depends only on the harmonic conformal class.
\end{proof}

Our first goal is to show that for given values of $A$ and $C$, the supremum in the definition of $\alpha_C(A)$ is attained.
This fact is the advantage of considering the enlarged space $\ol{[g]_h}$ over $[g]_h$ -- our sole motivation for introducing the generalized harmonic
conformal class was to prove the existence of a maximizer.  
\begin{thm}
\label{thm_maximizer}
Given $A>0,C\geq 1$, there exists $\ol g =u^{k} g \in \ol{[g]_h}$ satisfying
$$|\Sigma|_{\ol g} \leq A \qquad \text{ \emph{and} } \qquad u \leq C$$
that attains the supremum in Definition \ref{def_alpha}:
$$\min(\Sigma,\ol g) = \alpha_C(A).$$
Moreover, $|\Sigma|_{\ol g}=A$, provided $C$ is sufficiently large 
($C^{p} \geq \frac{A}{|\Sigma|_g}$).
\end{thm}
From now on, such $\ol g$ will be called a \emph{maximizer for} $\alpha_C(A)$ without further comment.  Unlike the case of $\mu(A)$,
we make no claim that a maximizer for $\alpha_C(A)$ is unique.
\begin{proof}
Fix $A>0,C\geq 1$.  Let $\{u_i^{k} g\}_{i=1}^\infty$ be a maximizing sequence for $\alpha_C(A)$ in $\ol{[g]_h}$.  That is, assume
\begin{equation}
|\Sigma|_{u_i^{k} g} \leq A, \quad u_i \leq C, \quad \text{ and }\quad
\min(\Sigma, u_i^{k} g) \nearrow \alpha_C(A).
\label{eq_increasing_seq}
\end{equation}
Let $f_i$ be the function on $\Sigma$
associated to $u_i$, so $f_i \leq C$ almost-everywhere on $\Sigma$.  Since
$$\int_\Sigma f_i^p dA = |\Sigma|_{u_i^{k} g} \leq A,$$
the sequence $\{f_i\}$ is bounded in $L^{p}(\Sigma)$, and thus has a weakly convergent
subsequence (of the same name, say) with limit $f \in L^{p}(\Sigma)$ (by the Banach--Alaoglu theorem \cite{reed_simon}).  Recall this means
that for all continuous functions $\phi$ on $\Sigma$,
\begin{equation}
\lim_{i \to \infty} \int_\Sigma f_i \phi dA = \int_\Sigma f \phi dA.
\label{eqn_weak_convergence}
\end{equation}
Redefining on a set of measure zero,
we may assume that $f \leq C$.
Let $u$ be the harmonic function associated to $f$, and let $\ol g = u^{k} g$, an element of $\ol{[g]_h}$.  By the maximum principle, $u \leq C$.
Since the $L^{p}$ norm is lower semi-continuous with respect to weak convergence, we have
that $\int_\Sigma f^{p}dA \leq A$, so that $|\Sigma|_{\ol g} \leq A$. In other words, $\ol g$ is a valid test metric for $\alpha_C(A)$. 
We claim that $\ol g$ is a maximizer for $\alpha_C(A)$, i.e., $\ol g$ has minimal enclosing area equal to $\alpha_C(A)$.

Let $S$ be a surface enclosing $\Sigma$ that is disjoint from $\Sigma$. From the definition of the minimal enclosing area, we have that for all $i$,
\begin{equation}
\label{eqn_min_test}
\min (\Sigma, u_i^{k} g) \leq |S|_{u_i^{k} g}.
\end{equation}
The left hand side converges to $\alpha_C(A)$ by assumption.  We also see from (\ref{eqn_assoc_harmonic}), the fact that the Poisson kernel
$y \mapsto K(x,y)$ is continuous for $x \not \in \Sigma$, and the definition of weak convergence (\ref{eqn_weak_convergence}) 
that $u_i \to u$ pointwise in the interior of $M$ as $i \to \infty$.  By harmonicity, this convergence is uniform on compact sets disjoint 
from $\Sigma$.  The surface $S$ is such a set, so
$$\lim_{i \to \infty} |S|_{u_i^{k} g} = \lim_{i \to \infty} \int_S u_i^{p} d\Hscr^{n-1} = \int_S u^{p} d\Hscr^{n-1} = |S|_{\ol g}.$$
Now, taking the limit $i \to \infty$ of (\ref{eqn_min_test}), we obtain for all $S$ disjoint from $\Sigma$:
$$\alpha_C(A) \leq |S|_{\ol g}.$$
Lemma \ref{lemma_push_out} below shows that to compute the minimal enclosing area, it is sufficient to consider only surfaces $S$ disjoint from the boundary.  It then follows that
$$\alpha_C(A) \leq \min(\Sigma,\ol g).$$
From the definition of $\alpha_C(A)$, the reverse inequality holds as well, so $\ol g$ is the desired maximizer for $\alpha_C(A)$.

Finally, we argue that we may assume $\ol g$ has boundary area equal to $A$, for $C$ large enough.  If not, suppose
that 
\begin{equation}
|\Sigma|_{\ol g} = \int_\Sigma f^{p} dA < A \qquad \text{and} \qquad C^{p} > \frac{A}{|\Sigma|_g}.
\label{eqn_area_C}
\end{equation}
These inequalities preclude
the possibility that $f \equiv C$ almost-everywhere.  For $t \in [0,1]$, consider the family of functions
$$f_t = f + t(C-f)$$
that interpolate between $f$ and $C$.  Note that $f \leq f_t \leq C$ for each $t$.  By the intermediate value theorem 
and (\ref{eqn_area_C}), there exists $s \in (0,1)$ such that
$$\int_\Sigma f_{s}^{p} dA = A.$$
Let $u_{s}$ be the harmonic function associated to $f_{s}$.  By construction, $u_{s}^{k} g$ is a valid test metric for $\alpha_C(A)$.
By the maximum principle, since $f_{s} \geq f$, we have that $u_{s} \geq u$ pointwise.  Then we have an inequality for the
minimal enclosing areas:
$$\min(\Sigma, u_{s}^{k} g) \geq \min(\Sigma, u^{k} g).$$
The left-hand side is at most $\alpha_C(A)$, and the right-hand side was already shown to equal $\alpha_C(A)$.  It follows that $u_{s}^{k}$
is a maximizer for $\alpha_C(A)$, and moreover its boundary area equals $A$.
\end{proof}

Now we prove a lemma used in the above construction of a maximizer for $\alpha_C(A)$.
\begin{lemma}
\label{lemma_push_out}
For the purposes of computing the minimal enclosing area (\ref{eqn_min_enclosing_area})
with respect to $\ol g=u^{k} g \in \ol{[g]_h}$, with $u \leq C$,
it is sufficient to consider only surfaces
$S$ that are disjoint from the boundary.
\end{lemma}
\begin{proof}

Let $S$ be any surface enclosing the boundary.  Let $X$ be a smooth, compactly supported vector field on $M$ such that $X|_\Sigma$
equals $\nu$, the unit normal vector field to $\Sigma$ pointing into $M$.  
For $t\geq 0$, let $\Phi_t:M\to M$ be the flow generated by $X$; note that $\Phi_t$ is a diffeomorphism onto its image, 
is the identity map outside a compact set, and maps $\Sigma=\partial M$ into the interior of $M$ for $t>0$.  In particular, for $t>0$, $\Phi_t(S)$
is a surface enclosing $\Sigma$ that is disjoint from the boundary.  To prove the lemma, we need only show that the area of $\Phi_t(S)$ with respect
to $\ol g$ varies continuously in $t$.

Observe that $|S|_{\ol g} < \infty$, since $u \leq C$ and $|S|_g$ is finite by our definition of surface. 
We decompose $S$ into
the disjoint, $\Hscr^{n-1}$-measurable sets $S \cap \Sigma$ and $S \sm \Sigma$.  First, consider $S \cap \Sigma$.  Reparametrizing the integral,
\begin{equation}
|\Phi_t(S\cap\Sigma)|_{\ol g}=\int_{\Phi_t(S \cap \Sigma)} u^{p} d\Hscr^{n-1} = \int_{S \cap \Sigma} (u \circ \Phi_t)^{p} d(\Phi_t^* \Hscr^{n-1}), 
\label{eqn_reparam_integral}
\end{equation}
where we have formed the pullback measure $\Phi_t^*\Hscr^{n-1}$ on $\Sigma$ of the measure $\Hscr^{n-1}$ on $\Phi_t(\Sigma)$:
$$\Phi_t^* \Hscr^{n-1}(E) := \Hscr^{n-1}(\Phi_t(E)),$$
for $E \subset \Sigma$ measurable.  Since $\Phi_t$ is a diffeomorphism, $\Phi_t^*\Hscr^{n-1}$ is absolutely continuous with respect to $\Hscr^{n-1}$, and we may write
$$d (\Phi_t^*\Hscr^{n-1}) = J_t d\Hscr^{n-1}$$
for a measurable function $J_t$ on $\Sigma$ that converges uniformly to 1 as $t \to 0^+$ (since $\Phi_t \to$ identity smoothly as $t \to 0^+$).
Next, $u \circ \Phi_t: \Sigma \to \R$ converges pointwise almost-everywhere to $f$ (see the remarks
preceding Definition \ref{def_ghcc}).  
Since $u \leq C$, the dominated convergence theorem allows us to evaluate the limit:
$$\lim_{t \to 0^+} \int_{S \cap \Sigma} (u \circ \Phi_t)^{p} d(\Phi_t^* \Hscr^{n-1}) 
  = \lim_{t \to 0^+} \int_{S \cap \Sigma} (u \circ \Phi_t)^{p} J_t d\Hscr^{n-1} = \int_{S \cap \Sigma} f^{p} d\Hscr^{n-1}.$$
The left-hand side is $\lim_{t \to 0^+} |\Phi_t(S\cap\Sigma)|_{\ol g}$ by (\ref{eqn_reparam_integral}); the right-hand side is $|S \cap \Sigma|_{\ol g}$.

The proof for $S \sm \Sigma$ is essentially the same: on $S \sm \Sigma$, $u \circ \Phi_t$ converges pointwise to $u$ and is dominated
by the integrable function $C$.
\end{proof}

We have established that the maximum for $\alpha_C(A)$ is attained, but we reiterate the point that the maximizer does not seem to satisfy a variational
principle that would allow us to automatically deduce regularity of this maximizer (see points (i)--(iii)
near the beginning of section \ref{sec_area_profile}).  In
the next section, we study maximizers in a regular case.  For now, we close this section by showing some nice properties satisfied by the function $\alpha(A)$.
\begin{prop}$\;$
\label{prop_alpha_continuous}
\begin{enumerate}[(i)]
\item $\alpha:\R^+ \to \R^+$ is nondecreasing, Lipschitz continuous, satisfying \mbox{$\alpha(A) \leq A$} for $A>0$.
\item There exists $A>0$ such that $\alpha(A) < A$.
\end{enumerate}
\end{prop}
The second statement rules out the possibility that $\alpha$ is the identity function $\R^+ \to \R^+$.
\begin{proof}$\,$\\
\emph{(i)  } 
The bounds $0 < \alpha(A) \leq A$ follow immediately from the definitions of $\alpha_C$ and $\alpha$, as does the fact that $\alpha$ is nondecreasing.  (One can show 
that $\alpha$ is strictly increasing, but we do not carry this out here.)

Let $0<A_1 < A_2$, and set $\eta = \left(\frac{A_1}{A_2}\right)^{1/p}$, a number in $(0,1)$.  Fix a constant $C$ so that $C^{p} > \frac{A_2}{|\Sigma|_g}$.
By Theorem \ref{thm_maximizer}, there exists a maximizer $\ol g_2 =u_2^{k} g \in \ol{[g]_h}$ for $\alpha_{C}(A_2)$ such that $|\Sigma|_{\ol g_2} = A_2$ and $u_2 \leq C$.  
Let $f_2 \in L^{p}(\Sigma)$ be the function on $\Sigma$ associated to $u_2$, and note that
$$\int_\Sigma \left(\eta f_2 \right)^{p} dA = A_1$$
by our choice of $\eta$.
Let $u_1$ be the harmonic function associated to $\eta f_2$, and let $\ol g_1= u_1^{k} g$. 
In particular, $\ol g_1$ measures the boundary area to be $A_1$, and $u_1 \leq u_2 \leq C$
(where the first inequality follows from the maximum principle).  Then $\ol g_1$ is a valid test metric
for $\alpha_{C}(A_1)$, so 
$$\alpha_{C}(A_1) \geq \min(\Sigma, \ol g_1),$$
by definition.  Next, for $x$ in the interior of $M$, by (\ref{eqn_assoc_harmonic})
\begin{align*}
u_1(x) &= \varphi(x) + \int_\Sigma K(x,y) \eta f_2(y) dA(y)\\
	&= \varphi(x) + \eta (u_2(x) - \varphi(x))\\
	&\geq \eta u_2(x),
\end{align*}
since $\eta < 1$ and $\varphi(x)>0$.  In particular, the minimal enclosing area for $\ol g_1 $ is at least $\eta^{p}$ times that for $\ol g_2$:
$$\min(\Sigma, \ol g_1) \geq \eta^{p} \min(\Sigma, \ol g_2).$$
But we chose $\ol g_2$ to have minimal enclosing area equal to $\alpha_C(A_2)$.
Putting our inequalities together, we have
$$\alpha_{C}(A_1) \geq \eta^{p} \alpha_C(A_2) = \frac{A_1}{A_2} \alpha_C(A_2).$$
Taking $\lim_{C \to \infty}$ of both sides and rearranging, we have
$$\frac{\alpha(A_1)}{A_1} \geq \frac{\alpha(A_2)}{A_2}$$
for all $A_1 < A_2$.
It follows that the function $A\mapsto \frac{\alpha(A)}{A}$ is non-increasing.  Combined with the fact that $\alpha(A)$ 
is non-decreasing and at most equal to $A$, one can readily check that $\alpha$ is Lipschitz continuous with Lipschitz constant at most 1.

\vspace{0.1in}
\noindent \emph{(ii)  } 
Given $\epsilon \in (0,1)$, we will produce a number $A > 0$ and a surface $S$ enclosing $\Sigma$ that has the 
property that 
\begin{equation}
\label{eqn_suff_cond}
|S|_{\ol g} < \epsilon |\Sigma|_{\ol g}, \quad  \text{for all metrics } \ol g \in \ol{[g]_h} \text{ with } |\Sigma|_{\ol g} = A.
\end{equation}
Assuming this to be the case, it follows that
\begin{align*}
\min(\Sigma,\ol g) &\leq |S|_{\ol g}\\
	 &< \epsilon|\Sigma|_{\ol g} = \epsilon A.
\end{align*}
In particular, letting $\ol g$ be a maximizer for $\alpha_C(A)$, then letting $C \to \infty$, we deduce
\mbox{$\alpha(A) \leq \epsilon A < A$.}  To complete the proof, we proceed to establish
(\ref{eqn_suff_cond}).

Without loss of generality (by rescaling), assume $|\Sigma|_g = 1$.
Observe that a harmonic function $u$ that is one at infinity with $L^p$ boundary data on $\Sigma$ can
be uniquely written as
$$u = \psi_\lambda:=\lambda \psi + \varphi,$$
where $\lambda > 0$ is a parameter, $\psi$ is a $g$-harmonic function on $M$ tending to zero at infinity with $L^p(\Sigma)$ boundary data
of $L^p$ norm equal to one.  As usual, $\varphi$ is harmonic, zero on $\Sigma$ and one at infinity.
We will also use the letter $\psi$ to denote the boundary data for $\psi$.  The point is that any metric in the generalized harmonic conformal class 
$\ol{[g]_h}$ with boundary area $A$ can be uniquely written in the form $g_\lambda = \psi_\lambda^{k} g$ for some $\psi$ as above and $\lambda = A^{1/p}$.
For now, take $\lambda > 0$ and $\psi$ with $\int_\Sigma \psi^p dA = 1$ to be arbitrary.


The Poisson kernel $K(x,y)$ is harmonic as a function of $x$, approaching zero at infinity. 
Identifying $x$ with an asymptotically flat coordinate chart, $K(x,y)$ is $O(r^{-n+2})$ in $x$ for large $r=|x|$.  (This decay
is independent of $y \in \Sigma$, since $\Sigma$ is compact.)  Then any harmonic function $\psi$ as above satisfies:
\begin{align*}
\psi(x) &= \int_\Sigma K(x,y) \psi(y) dA(y)\\
		&\leq \frac{c}{r^{n-2}} \int_\Sigma \psi(y) dA(y)\\
		&\leq \frac{c}{r^{n-2}} \left(\int_\Sigma \psi^{p} dA\right)^{\frac{1}{p}} |\Sigma|_g^{\frac{n}{2(n-1)}} &&\text{(by H\"{o}lder's inequality)}\\
		&= \frac{c}{r^{n-2}} 
\end{align*}
where $c>0$ is a constant depending only on $(M,g)$ but not on $\psi$.

Let $\epsilon \in (0,1)$ be given.  The $g$-area of a coordinate sphere $S_r$ in $M$ is asymptotic to $\omega_{n-1}r^{n-1}$ and is therefore 
less than $2 \omega_{n-1} r^{n-1}$ for $r$ sufficiently large. 
In particular, the quantity
$$\int_{S_r} \psi^{p} dA \leq 2\omega_{n-1} r^{n-1} \cdot \frac{c^{p}}{r^{2(n-1)}} $$
can be made less than $\epsilon^2$ by choosing $r$ sufficiently large, \emph{independently of} $\psi$.
Fix such a value of $r$, and let $S=S_r$.

By construction, the area of $\Sigma$ with respect to any $g_\lambda$ as above is $\lambda^{p}$.  
Let us compute the area of $S$ in the metric $g_\lambda$:
\begin{align*}
|S|_{g_\lambda} &= \int_{S} \psi_\lambda^{p} dA\\
	&= \int_S(\lambda \psi + \varphi)^{p} dA.
\end{align*}
Factoring out $\lambda^{p}$, applying the Minkowski inequality, and using $0\leq \varphi \leq 1$ we have
\begin{align*}
|S|_{g_\lambda} &\leq \lambda^{p} \left( \left(\int_S \psi^{p} dA\right)^{1/{p}} + \lambda\inv \left(\int_S \varphi^{p} dA
\right)^{1/{p}}\right)^{p}\\
&\leq \lambda^{p} \left( \epsilon^{2/{p}} +  \lambda\inv |S|_g^{1/{p}}\right)^{p}.
\end{align*}
Then for some $\lambda$ sufficiently large,
$$|S|_{g_\lambda} \leq \epsilon \lambda^{p} = \epsilon |\Sigma|_{g_\lambda}$$
for all choices of $\psi$.  Since every metric $\ol g \in \ol{[g]_h}$ can be written as 
$g_\lambda$ for some $\lambda$ and $\psi$, we have shown (\ref{eqn_suff_cond}) with $A=\lambda^p$, completing the proof.
\end{proof}

\begin{cor}
\label{cor_alpha_A}
Suppose that $\alpha(A) < A$ for some value of $A$.  Then $\alpha(B) < B$ for all $B \geq A$.  Moreover,
$$\lim_{A \to \infty} \frac{\alpha(A)}{A} = 0.$$
\end{cor}
\begin{proof}
Both statements follow from the proof of Proposition \ref{prop_alpha_continuous}.  For the first, we showed that $\frac{\alpha(A)}{A}$ is non-increasing as a function of $A$.  For the second, we also showed that given $\epsilon > 0$, there exists $A>0$ large so that
$\frac{\alpha(A)}{A} < \epsilon$.
\end{proof}

\section{Properties of maximizers for the area profile function: the smooth case}
\label{sec_properties}
Continuing the previous section, suppose that $\alpha(A)<A$ and that $\ol g = u^{k} g \in [g]_h$ is a maximizer for $\alpha_C(A)$
as in Theorem \ref{thm_maximizer}.  At this point, we know only that the boundary data $f$ for the conformal factor $u$ is a nonnegative
function in $L^{\infty}(\Sigma)$.  In \cite{thesis}, we worked directly with the poor regularity to prove results on the geometry of $\ol g$ under some additional technical
hypotheses, and with a slightly different definition of $\alpha(A)$.  Rather than pursue this approach here, our present purpose is
to prove a result regarding $\ol g$ in the case where the maximizing metric is \emph{a priori} assumed to be smooth.  This suggests a heuristic
for what we expect to occur in general (c.f. Conjecture \ref{conj_rpi}).


The main idea of the following theorem is that if $\ol g$ is regular, then the outermost minimal area enclosure touches the boundary only on a set of small measure and has uniformly small mean curvature with respect to $\ol g$.
\begin{thm}  Let $(M,g)$ be asymptotically flat of dimension $3 \leq n \leq 7$ with compact boundary $\Sigma$.  
Suppose $\alpha(A)<A$ and $C^p > \max\left(\frac{A}{|\Sigma|_g},1\right)$.  Let $\ol g = u^{k} g$ be a maximizer for $\alpha_C(A)$ given by Theorem \ref{thm_maximizer}, and assume that $u$ is smooth and positive on 
$\Sigma$.   Let $\tilde \Sigma$ be the outermost minimal area enclosure of $\Sigma$ with respect to $\ol g$.  Then 
\begin{enumerate}[(i)]
 \item  $\Hscr^{n-1}(\tilde \Sigma \cap \Sigma) \leq A C^{-p}$ (where $\Hscr^{n-1}$ is Hausdorff $(n-1)$-measure on
$(M,g$)), and
 \item the mean curvature $\ol H$ of $\tilde \Sigma \cap \Sigma$ with respect to $\ol g$ is bounded (pointwise almost-everywhere) between $0$ and  
  $\eta_0 C^{-\frac{2}{n-2}}$, for some constant $\eta_0$ depending only on $g$.
\end{enumerate}
\label{thm_mean_curv}
\end{thm}
Note that $\tilde \Sigma$ exists because $u$, and therefore $\ol g$, is smooth by assumption (see the beginning of section \ref{sec_area_profile}). Moreover, $\tilde \Sigma$ has
nonnegative mean curvature with respect to $\ol g$, or else an outward variation would produce a surface of less $\ol g$-area.  Since $\tilde \Sigma \sm \Sigma$
is a minimal surface for $\ol g$, the theorem states that $\tilde \Sigma$ has uniformly small mean curvature (depending on $C$).  The significance of $\alpha(A) < A$ is that it prevents $\Sigma$ from being its own outermost minimal area enclosure: $\Sigma$ has area
$A$ and $\tilde \Sigma$ has area $\alpha_C(A)\leq \alpha(A)$.

\begin{proof}
In the case that $\tilde \Sigma \cap \Sigma$ is empty or has zero $\Hscr^{n-1}$-measure, we are done.  Assume otherwise; we aim to show that
\begin{equation}
u = C \text{ almost-everywhere on the set } \tilde \Sigma \cap \Sigma.
\label{eqn_u_equiv_C}
\end{equation}
Supposing that (\ref{eqn_u_equiv_C}) holds, we complete the proof.  First,
\begin{align*}
A &= |\Sigma|_{\ol g} \geq |\tilde \Sigma \cap \Sigma|_{\ol g} =\int_{\tilde \Sigma \cap \Sigma} u^{p} d\Hscr^{n-1}.
\end{align*}
But by (\ref{eqn_u_equiv_C}),this reduces to $A \geq C^{p}\, \Hscr^{n-1}(\tilde \Sigma \cap \Sigma)$, proving (i).

Recall that $\tilde \Sigma$ is a $C^{1,1}$ surface and has mean curvature $\ol H$ (with respect to $\ol g$) defined almost-everywhere and nonnegative.  
The set $\tilde \Sigma \sm \Sigma$ is a smooth hypersurface with $\ol H = 0$.  (These regularity assertions require $n\leq 7$.)
On the set $\tilde \Sigma \cap \Sigma$, $\ol H$ agrees with the mean curvature $\ol H_\Sigma$ of 
$\Sigma$ with respect to $\ol g$ (see (1.15) of \cite{imcf}).  Then for almost all $y \in \tilde \Sigma \cap \Sigma$,
\begin{equation}
\ol H(y) = \ol H_\Sigma(y) = u(y)^{-\frac{2}{n-2}} H_\Sigma(y) + \frac{2(n-1)}{n-2} u(y)^{-\frac{n}{n-2}} \partial_\nu u(y), 
\end{equation}
having used the law for the transformation of mean curvature under conformal changes.  Here $H_\Sigma$ is the mean curvature of $\Sigma$
with respect to $g$.  Using (\ref{eqn_u_equiv_C}), we have $u(y)=C$; moreover, by the maximum principle, $\partial_\nu(u)(y) < 0$, since
$u$ attains its global maximum value of $C$ at $y$.  Let $\eta_0 = \|H\|_{C^0(\Sigma)}$, so that
$$0 \leq \ol H(y) \leq \eta_0 C^{-\frac{2}{n-2}},$$
for almost all $y \in \tilde \Sigma \cap \Sigma$, proving (ii).

To complete the proof, we use the failure of (\ref{eqn_u_equiv_C}) to construct a valid ``variation'' within $\ol{[g]_h}$ that increases the minimal
enclosing area, contradicting the assumption that $\ol g$ is a maximizer for $\alpha_C(A)$.  If (\ref{eqn_u_equiv_C}) fails,
there exists $E \subset \tilde \Sigma \cap \Sigma$ of positive $\Hscr^{n-1}$-measure and a constant $\epsilon>0 $ such that
$$u \leq C-\epsilon \quad \text{ on } E.$$
Let $\chi: \Sigma \to \R$ be the characteristic function of $E$.  For $t \geq 0$, consider the following
family of functions on $\Sigma$:
$$f_t = f\left(1 + t\left(\chi - a\right)\right),$$
where $a$ is the constant $a= \frac{\int_\Sigma f^{p} \chi dA}{A}$.
For all $t$ sufficiently small, say $t \in [0,t_0)$, $f_t$ is positive and bounded above by $C$ (using the definition of $E$ and $\chi$).

We consider $f_t$ as the boundary data for harmonic functions $u_t$ that approach one at infinity. A formula for $u_t$ is
$$u_t = u + t(w-a(u-\varphi)),$$
where $w$ is harmonic, zero at infinity, with boundary data $f\chi$ on $\Sigma$.
Then $\ol g_t := u_t^{k} g$ is a smooth path of metrics in $\ol{[g]_h}$ passing through $\ol g$ at $t=0$. These metrics are valid test metrics 
for $\alpha_C(A)$, modulo the fact that the boundary area $A(t)= |\Sigma|_{\ol g_t}$ of $\Sigma$ is not necessarily $\leq A$; we address this issue later.
For now, observe that $A(t)$ is stationary to first order at $t=0$ by our choice of the constant $a$:
\begin{equation*}
A'(0) = \frac{d}{dt}\Big|_{t=0} \int_\Sigma f_t^{p} dA = \int_\Sigma p f^{p} \left(\chi-a\right) dA = 0.
\end{equation*}
We proceed to show that the minimal enclosing area $\min(\Sigma, \ol g_t)$ is increasing near $t=0$.  Let $S$ be any minimal area enclosure of $\Sigma$ with respect to 
$\ol g$ (possibly $\tilde \Sigma$ itself).  Since $\tilde \Sigma$ is the \emph{outermost} minimal area enclosure, we see that 
\begin{equation}
\label{eqn_omae_contains}
\tilde \Sigma \cap \Sigma \subset S \cap \Sigma.
\end{equation}
We estimate the rate of change of the area of the fixed surface $S$ along the path of metrics $\ol g_t$:
\begin{align*}
\frac{1}{p} \frac{d}{dt}\Big|_{t=0} |S|_{\ol g_t} &= \frac{1}{p} \frac{d}{dt}\Big|_{t=0} \left(\int_{S \cap \Sigma} f_t^{p} dA + \int_{S \sm \Sigma} u_t^{p} dA\right)\\
	&= \int_{S \cap \Sigma}  f^{p} (\chi -a) + \int_{S \sm \Sigma} u^{p-1} (w-au+a\varphi)\, dA\\
	&> \int_{\tilde \Sigma \cap \Sigma} f^{p} \chi \, dA  - a  \left( \int_{S \cap \Sigma} f^{p} \, dA + \int_{S \sm \Sigma} u^{p} \, dA\right) \\
        &= aA - a |S|_{\ol g}\\
        &\geq a(A - \alpha(A)).
\end{align*}
On the third line, we used (\ref{eqn_omae_contains}) and the fact that $u, w,$ and $\varphi$ are positive on the set $S \sm \Sigma$
(by the maximum principle).  On the last two lines, we used  the fact that $\chi$ is supported in $E \subset \tilde \Sigma \cap \Sigma$, the definition of $a$,
and the fact that $|S|_{\ol g} = \min(\Sigma, \ol g) = \alpha_C(A) \leq \alpha(A)$.  The last term is positive, since $\alpha(A)<A$ by
hypothesis.
This shows that the rate of change of areas of all minimal area enclosures is uniformly positive at $t=0$, proving that 
$\min(\Sigma, \ol g_t)$ is increasing near $t=0$.

To remedy the fact that $\ol g_t$ does not fix the boundary area for $t>0$, define
$\tilde f_t = \frac{A^{1/p}}{A(t)^{1/p}} f_t$,
and consider the harmonic functions $\tilde u_t$ with boundary data $\tilde f_t$.  Then by construction, the metrics 
$\tilde g_t = \tilde u_t^{k} g$ have boundary area equal to $A$,
and are valid test metrics for $\alpha_C(A)$.  Since $A'(0)=0$, first derivative computations at $t=0$ agree for $\ol g_t$ and $\tilde g_t$; 
by the above, $\min(\Sigma, \tilde g_t)$ is increasing near $t=0$.  This contradicts the assumption that $\ol g_0 = \ol g$ was a maximizer 
of $\alpha_C(A)$.  

%


\end{proof}
In section \ref{sec_applications}, we propose a conjecture regarding maximizers $\ol g$ of $\alpha_C(A)$ in general, without \emph{a priori} assumptions
on regularity.

\section{Examples}
\label{sec_examples}
Suppose $M$ is $\R^n$ minus the unit open ball centered at the origin, equipped with the flat metric $g$.  The function $\varphi(x) = 1-\frac{1}{|x|^{n-2}}$
is harmonic, vanishes on $\Sigma=\partial M$, and approaches one at infinity.  Now it is straightforward to compute that $I_1 = -2$ and $I_2 = 2$, so that
$$\mu(A) = -2 + 2 \left(\frac{A}{\omega_{n-1}}\right)^{1/p}$$
by Theorem \ref{thm_mu}.  Recall from section \ref{sec_definitions} that $(M,g)$ is in the same harmonic conformal class as the Schwarzschild metric of mass 2.
See figure \ref{fig_mu_plot} for a plot of $\mu$ for $n=3$.
\begin{figure}[ht]
\caption{Plot of mass profile function}
\begin{center}
\includegraphics[scale=1.0]{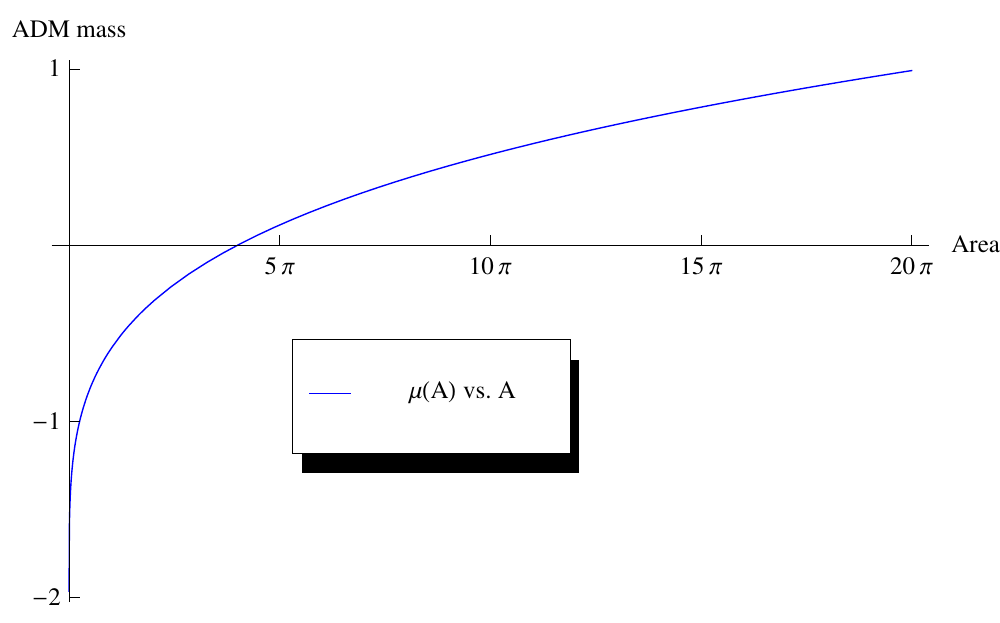} 
\end{center}
\scriptsize{Above is a plot of $\mu(A)$ vs. $A$ for the harmonic conformal class of $\R^3$ minus a unit ball.}
\label{fig_mu_plot}
\end{figure}


Next, to address $\alpha(A)$ for harmonic conformal class of $(M,g)$, let $u_A$ be the spherically-symmetric harmonic function that is one 
at infinity such that $(\R^n, u_A^{k} g)$ has boundary area equal to $A$, given explicitly by:
$$u_A(x) = 1 + \frac{\left(\frac{A}{\omega_{n-1}}\right)^{1/p}-1}{|x|^{n-2}}.$$
To aid with the discussion, we recall that formula (\ref{eqn_schwarz_metric}) extends
to define a metric on $\R^n \sm \{0\}$ with two asymptotically flat ends, with a reflection symmetry across the Euclidean sphere of radius $\left(\frac{m}{2}\right)^{\frac{1}{n-2}}$,
called the horizon.  The horizon is a minimal surface for the Schwarzschild metric.  Let $m=2\left(\frac{A}{\omega_{n-1}}\right)^{1/p}-2$, so we see that $(M, u_A^{k} g)$ is isometric to a 
subset of the two-ended Schwarzschild manifold of mass $m$.  

If $m < 2$ the boundary $\Sigma$ (the Euclidean unit sphere) is its own outermost minimal area enclosure, and therefore has area $A$, essentially because $(M,u_A^{k} g)$ excludes the horizon.  
If $m > 2$, $(M,u_A^{k} g)$ includes the horizon of the Schwarzschild manifold, which is the outermost minimal area enclosure $\Sigma$ and has area
$(2m)^{\frac{n-1}{n-2}} \omega_{n-1}$.  In the borderline case $m=2$, the boundary of $M$ agrees with the horizon.
These observations allow us to give a formula
for the minimal enclosing area as a function of $A$:
 $$\min(\Sigma, u_A^{k} g) = \begin{cases}
				  A, & \text{if } A \leq 2^{p} \omega_{n-1}\\
				  \left(\left(\frac{A}{\omega_{n-1}}\right)^{1/p} - 1\right)^{\frac{n-1}{n-2}} 2^{p}\omega_{n-1} , & \text{if } A > 2^{p} \omega_{n-1}
                              \end{cases}.$$
We conjecture that the spherically symmetric metric $u_A^{k} g$ is a maximizer for $\alpha_C(A)$, for all $C \geq \left(\frac{A}{\omega_{n-1}}\right)^{1/p}$.
This would immediately imply that $\alpha(A)$ is given by the above formula for $\min(\Sigma, u_A^{k} g)$.
Without proof, we remark that the above spherically-symmetric metrics
$u_A^{k} g$ are \emph{local} maxima for the minimal enclosing area among metrics in $\ol{[g]_h}$ that have boundary area $A$.  
See figure \ref{fig_alpha_plot} for a plot of this conjectured form of $\alpha$.
\begin{figure}[ht]
\caption{Plot of area profile function}
\begin{center}
\includegraphics[scale=1.0]{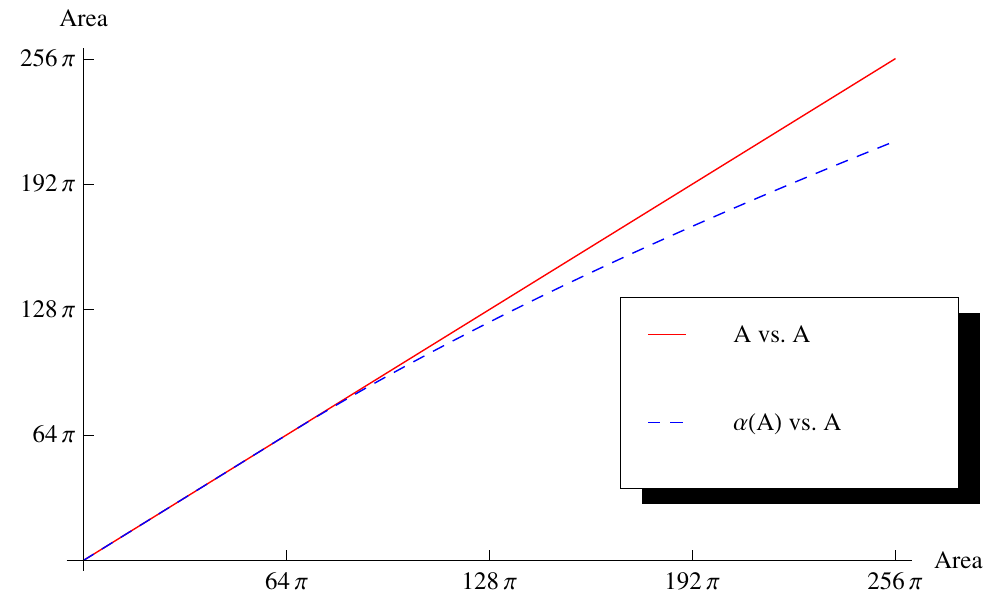} 
\end{center}
\scriptsize{Above is a plot of the conjectured form of $\alpha(A)$ vs. $A$ for the harmonic conformal class of $\R^3$ minus a unit ball, overlaid
with a plot of $A$ vs. $A$ for comparison.  The two functions agree precisely on the interval $[0,64\pi]$.}
\label{fig_alpha_plot}
\end{figure}
If $\alpha(A)$ does indeed have the above form in this example, then we remark that we have equality
$$\mu(A) = \frac{1}{2} \left(\frac{\alpha(A)}{\omega_{n-1}}\right)^{\frac{n-2}{n-1}} \quad \text{ for } A > 2^{p} \omega_{n-1}.$$
Finally, we point out that the metrics $u_A^{k} g$ behave consistently with Theorem \ref{thm_mean_curv}: in the case $\alpha(A)<A$, the surface $\tilde \Sigma$
for $u_A^{k} g$ is disjoint from $\Sigma$ and therefore has zero mean curvature.

\section{Conjectured applications}
\label{sec_applications}
In this section we present a natural conjecture regarding the function $\alpha(A)$.  Assuming this conjecture, we deduce statements relating
the functions $\mu(A)$, $\alpha(A)$ and the numerical invariants $I_1$ and $I_2$.  One consequence is a general estimate of the ADM mass
of an asymptotically flat manifold of nonnegative scalar curvature with compact boundary.

First we recall the Riemannian Penrose inequality, proved as stated below by Bray \cite{bray_RPI} (for dimension $n=3$) 
and later by Bray and Lee \cite{bray_lee} (for $3 \leq n\leq 7$).  Huisken and Ilmanen gave a proof for $n=3$, with $A$
replaced by the area of the largest connected component of $\partial M$ \cite{imcf}.
\begin{thm}
Let $(M^n,g)$ be asymptotically flat of dimension $3 \leq n \leq 7$ with nonnegative scalar curvature.  Suppose the boundary
$\Sigma = \partial M$ has area $A$, zero mean curvature, and every surface enclosing $\Sigma$ has area strictly greater than $A$.  Then
$$m_{ADM}(g) \geq \frac{1}{2}\left(\frac{A}{\omega_{n-1}}\right)^{\frac{n-2}{n-1}}.$$
Moreover, if equality holds and $M$ is a spin manifold (or if $n=3$), then $(M,g)$ is isometric to the Schwarzschild manifold of mass $m_{ADM}(g)$.
\end{thm}
We remark that the harmonic conformal class was a crucial element in both the Bray and Bray--Lee proofs.  For the remainder of the
article we assume, unless noted otherwise, that $(M,g)$ has dimension $3 \leq n \leq 7$ and nonnegative scalar curvature (but not necessarily minimal boundary).

Consider a maximizer $\ol g_C$ for $\alpha_C(A)<A$ with $C$ large.  If we assume $\ol g_C$ is smooth, then based on Theorem \ref{thm_mean_curv}, 
we see that the outermost minimal area enclosure $\tilde \Sigma_C$ is ``close to'' a minimal surface: the mean curvature is uniformly
 bounded by a constant times $C^{-\frac{2}{n-2}}$.  By construction, $\tilde \Sigma_C$ has less area than any surface that encloses it.
Together with the Riemannian Penrose inequality, this behavior suggests that the ADM mass of $\ol g_C$ ought to be bounded from below in
terms of $|\tilde \Sigma_C|_{\ol g_C} = \min(\Sigma, \ol g_C)$ in the limit $C \to \infty$.  


\begin{conj}
\label{conj_rpi} Let $(M,g)$ be an asymptotically flat $n$-manifold, $3 \leq n \leq 7$, of nonnegative scalar curvature, with nonempty, smooth, compact boundary $\Sigma$.
Fix $A > 0$ for which $\alpha(A)<A$.  Let $\ol g_C$ be a maximizer for $\alpha_C(A)$.  Then
$$\lim_{C \to \infty} m_{ADM}(\ol g_C) \geq \lim_{C \to \infty} \frac{1}{2}\left(\frac{\min(\Sigma, \ol g_C)}{\omega_{n-1}}\right)^{\frac{n-2}{n-1}}.$$
In other words, the Riemannian Penrose inequality holds for $(M,\ol g_C)$ in the limit $C \to \infty$.
\end{conj}
Assuming the conjecture, we prove some consequences.  
\begin{prop}
\label{thm_mu_alpha}
Assume that Conjecture \ref{conj_rpi} is true.  For all values of $A>0$ that satisfy $\alpha(A) < A$, we have the following inequality for the mass and area profile functions:
$$\mu(A) \geq \frac{1}{2}\left(\frac{\alpha(A)}{\omega_{n-1}}\right)^{\frac{n-2}{n-1}}.$$
\end{prop}
\begin{proof}
Suppose $\alpha(A) < A$.  For all $C$ sufficiently large, $\min(\Sigma, \ol g_C) = \alpha_C(A)$ by Theorem \ref{thm_maximizer}.  Since
$\alpha(A) = \lim_{C \to \infty} \alpha_C(A)$, Conjecture \ref{conj_rpi} can be written:
\begin{equation}
\lim_{C \to \infty} m_{ADM}(\ol g_C) \geq \frac{1}{2}\left(\frac{\alpha(A)}{\omega_{n-1}}\right)^{\frac{n-2}{n-1}}.
\label{eqn_mass_C}
\end{equation}
We claim the left-hand side is at most $\mu(A)$.  To see this, note that a harmonic function $u$ with $L^{p}$ boundary data can be approximated
in $L^{p}(\Sigma)$ norm by a harmonic function $u_\epsilon$ with smooth boundary data so that the ADM masses of $u^{k} g$ and $u_\epsilon^{k} g$ differ
by less than $\epsilon$.  Thus, the value of $\mu(A)$ 
in Definition \ref{def_mu} is unchanged if the supremum is taken over the generalized harmonic conformal class.  In
particular, $\ol g_C$ can be viewed as a valid test metric for $\mu(A)$, so we have
$$\mu(A) \geq m_{ADM}(\ol g_C)$$
for each $C$.  Taking the limit $C \to \infty$ completes the proof.
\end{proof}
We emphasize the point that while $\mu(A)$ is determined solely from the numerical invariants $I_1$ and $I_2$, $\alpha(A)$ involves
much more of the global geometry of $(M,g)$ -- the areas of hypersurfaces.  One interesting immediate consequence is the following
upper bound for the minimal enclosing area.  If $\alpha(A)< A$ (which holds for all $A>0$ sufficiently large by Proposition
\ref{prop_alpha_continuous} and Corollary \ref{cor_alpha_A}), then for all metrics $g' \in [g]_h$ with boundary area at most $A$:
$$\mu(A) \geq \frac{1}{2}\left(\frac{\min(\Sigma,g')}{\omega_{n-1}}\right)^{\frac{n-2}{n-1}},$$
as follows from the previous proposition and the definition of $\alpha$. By Theorem \ref{thm_mu}, the left-hand side can be 
computed explicitly in terms of $I_1$, $I_2$, and $A$.  This is an example of $I_1$ and $I_2$ giving control on
the geometry of metrics in $[g]_h$.

Next, we prove an inequality for the numerical invariants $I_1$ and $I_2$.
\begin{thm} 
\label{thm_I1_I2}
Assume Conjecture \ref{conj_rpi}.  Then the numerical invariants defined in Lemmas \ref{lemma_I1} and \ref{lemma_I2} satisfy:
$$I_1 + I_2 \geq 0.$$
\end{thm}
Recall that by definition, $I_2 > 0$.
\begin{proof}
Let $\epsilon > 0$ be given.  As a consequence of Proposition \ref{prop_alpha_continuous}, Corollary \ref{cor_alpha_A}, and the intermediate
value theorem, there exists $A>0$ such that
\begin{equation}
\label{eqn_A_epsilon}
A - \alpha(A) = \epsilon.
\end{equation}
By (\ref{eqn_mass_C}) (which uses the conjecture), there exists $C>0$ sufficiently large so that a maximizer $\ol g=u^{k} g$ of $\alpha_C(A)$ satisfies:
$$m_{ADM}(\ol g) \geq \frac{1}{2}\left(\frac{\alpha(A)}{\omega_{n-1}}\right)^{\frac{n-2}{n-1}} - \epsilon.$$
The left-hand side may be computed using formula (\ref{eqn_conf_mass}), and the right hand side with (\ref{eqn_A_epsilon}):
$$m_{ADM}(g) \geq \frac{1}{2}\left(\frac{A-\epsilon}{\omega_{n-1}}\right)^{\frac{n-2}{n-1}} + \frac{2}{(n-2)\omega_{n-1}} \int_{S_r} \partial_\nu u\, dA - \epsilon.$$
Applying Lemma \ref{lemma_compute_term} (which still holds for $L^{p}$ boundary data), we have
$$m_{ADM}(g) \geq \frac{1}{2}\left(\frac{\int_\Sigma f^{p} dA-\epsilon}{\omega_{n-1}}\right)^{\frac{n-2}{n-1}} - 2\int_{\Sigma} Vf \, dA +2C_g(\Sigma)- \epsilon,$$
where $f$ is the function on $\Sigma$ associated to $u$.  By taking the infimum of the above over all nonnegative functions
$f$ in $L^{p}(\Sigma)$ and letting $\epsilon \to 0$, we obtain:
\begin{equation}
m_{ADM}(g) \geq 2C_g(\Sigma) + \inf_{f \in L^{p}(\Sigma), f \geq 0} \left\{\frac{1}{2}\left(\frac{\int_\Sigma f^{p} dA}{\omega_{n-1}}\right)^{\frac{n-2}{n-1}} - 2\int_{\Sigma} Vf \, dA\right\}.
\label{eqn_mass_ineq_prelim}
\end{equation}
The term inside the braces can be viewed as a strictly convex functional on functions $f \in L^{p}(\Sigma)$.  Using an Euler--Lagrange approach,
one can compute that the unique global minimum of the functional is attained by
$$f(x) = 2\left(\omega_{n-1}\right)^{\frac{n-2}{n-1}} \left(\int_\Sigma V^{\frac{2(n-1)}{n}}dA\right)^{\frac{1}{n-1}}  V(x)^{\frac{n-2}{n}},$$
a smooth function on $\Sigma$ (see Chapter 5 of \cite{thesis} for the full details in the case $n=3$).  
Using the formula $V = \frac{1}{(n-2)\omega_{n-1}} \partial_\nu \varphi$ and our expression for the minimizer $f$, we deduce from
(\ref{eqn_mass_ineq_prelim}) that
\begin{equation}
\label{eqn_mass_estimate}
m_{ADM}(g) \geq 2C_g(\Sigma) - \frac{2}{(n-2)^2} \left(\frac{1}{\omega_{n-1}} \int_\Sigma (\partial_\nu \varphi)^{\frac{2(n-1)}{n}} dA\right)^{\frac{n}{n-1}}.
\end{equation}
Recalling the definitions of the numerical invariants $I_1$ and $I_2$ from Lemmas \ref{lemma_I1} and \ref{lemma_I2}, 
we see that the above inequality is equivalent to the statement $I_1 + I_2 \geq 0$.
\end{proof}
One interpretation of (\ref{eqn_mass_estimate}) is a general estimate for the ADM mass of an asymptotically flat manifold of nonnegative scalar curvature with
compact boundary.  In the case that the Riemannian Penrose inequality applies to $(M,g)$, inequality (\ref{eqn_mass_estimate}) is weaker.  However, we emphasize that (\ref{eqn_mass_estimate}) requires 
no assumptions on the boundary geometry, such as minimality.  

\subsection{Zero area singularities}
We give one final interpretation of Theorem \ref{thm_I1_I2}:
\begin{cor}
\label{cor_rzi}
Assume Conjecture \ref{conj_rpi}.  Let $\varphi$ be the $g$-harmonic function that vanishes on $\Sigma$ and approaches one at infinity.  Then
the asymptotically flat metric $g'=\varphi^{k} g$ (which is singular on $\Sigma$) satisfies the mass estimate:
\begin{equation}
m_{ADM}(g') \geq -  \frac{2}{(n-2)^2} \left(\frac{1}{\omega_{n-1}} \int_\Sigma (\partial_\nu \varphi)^{\frac{2(n-1)}{n}} dA\right)^{\frac{n}{n-1}}.
\label{eqn_rzi}
\end{equation}
\end{cor}
The proof follows from equation (\ref{eqn_mass_estimate}) and the observation that \mbox{$m_{ADM}(g) - 2C_g(\Sigma)$} 
equals $m_{ADM}(g')$ by formula (\ref{eqn_conf_mass_infinity}).  The reason that $g'$ is singular on the boundary is that the conformal factor $\varphi$ vanishes there.  This type of metric singularity is 
an example of a \emph{zero area singularity}, or ZAS, which we now describe.  Following \cites{bray_npms, zas, robbins}, in a manifold with smooth metric on the interior (but not necessarily
on the boundary),  a boundary component $S$ is said to be a zero area singularity if for all sequences of surfaces $\{S_n\}$ converging in the $C^1$ sense
to $S$, the areas of the $S_n$ converge to zero.  Metrics $g'= \varphi^{k} g$ as in the corollary have a ZAS on each boundary component.  Also note that $g'$ has nonnegative scalar curvature because $g$ does and $\varphi$ is harmonic.

The motivating example of a manifold with a zero area singularity is the Schwarzschild metric of negative mass: the metric given in equation (\ref{eqn_schwarz_metric})
with $m<0$ on $\R^n$ minus a ball of radius $\left(\frac{|m|}{2}\right)^{\frac{1}{n-2}}$.  


It is true but not immediately obvious that the right-hand side of (\ref{eqn_rzi}) is intrinsic to
the singularities, in that it depends only on the geometry of $g'$ in any neighborhood of $\Sigma$ (and not
on the data $(g,\varphi)$).
This number is suggestively called the \emph{mass} of $\Sigma$ (or ZAS mass), and equals $m$ for the Schwarzschild
metric of mass $m < 0$.  In general, for ZAS that do not 
necessarily arise from metrics of the form $g'=\varphi^{k} g$, it is still possible to define a meaningful notion
of ZAS mass, a number in $[-\infty,0]$.  Corollary \ref{cor_rzi} implies the statement: in manifolds of nonnegative
scalar curvature that contain zero area singularities $\Sigma$, the ADM mass is bounded below by the ZAS mass:
\begin{equation}
\label{eqn_adm_zas}
m_{ADM} \geq m_{ZAS}(\Sigma). \qquad\qquad\text{(conjectured})
\end{equation}
For a more thorough discussion, we refer the reader to \cites{bray_npms, zas, robbins}.

There are two cases in which inequality (\ref{eqn_adm_zas}) is firmly established without the use of Conjecture \ref{conj_rpi}.  
First, if $n=3$ and $\Sigma=\partial M$ is connected, Robbins \cite{robbins} proved the inequality using weakly-defined inverse mean
curvature flow as developed by Huisken and Ilmanen \cite{imcf}.  If $\Sigma$ is disconnected, however, inverse mean curvature flow
yields no such inequality, not even a weaker version.  The second case for which Corollary \ref{cor_rzi} is known is that in which the 
harmonic conformal class $[g]_h$ contains a metric for which the hypotheses of the Riemannian Penrose inequality hold.
If such a metric exists, Bray showed inequality (\ref{eqn_rzi}) directly from the Riemannian Penrose inequality (c.f. \cites{bray_npms, zas}).  However,
in general, $[g]_h$ need not contain such a metric (see Chapter 2 of \cite{thesis}).


In closing, we make a connection with the positive mass theorem (PMT) of Schoen and Yau \cite{schoen_yau}, proved also for spin manifolds by Witten \cite{witten}. 
Note that the PMT was a key ingredient in the Bray and Bray--Lee proofs of the Riemannian Penrose inequality.
\begin{thm}[Positive mass theorem] 
\label{thm_pmt}
Let $(M,g)$ be a complete, asymptotically flat 
Riemannian $n$-manifold without boundary, with either $3 \leq n \leq 7$ or $M$ a spin manifold.  If $(M,g)$ has nonnegative scalar curvature, then the ADM mass is nonnegative,
and zero if and only if $(M,g)$ is isometric to $\R^n$ with the flat metric.
\end{thm}
We can view (\ref{eqn_adm_zas}) as a generalization of the PMT: metrics with ZAS are generally incomplete, so 
Theorem \ref{thm_pmt} does not apply.  If we interpret the ZAS mass as quantifying the defect due to the presence of singularities
in terms of their local geometry, then inequality (\ref{eqn_adm_zas}) gives a lower bound for the ADM mass as the size of this defect.
We emphasize that (\ref{eqn_adm_zas}) is unproven in general.  A case of particular interest is when $(M,g)$ contains only ZAS of zero mass: inequality (\ref{eqn_adm_zas}) would establish nonnegativity of the ADM mass.

\end{document}